\documentclass[twoside]{article}

\usepackage[accepted]{aistats2019}
\usepackage[utf8]{inputenc}
\usepackage[T1]{fontenc}
\usepackage{booktabs}
\usepackage{multirow}
\usepackage{nccmath}

\usepackage{amsmath}
\usepackage{amsthm}
\usepackage{amssymb}
\usepackage{empheq}
\usepackage{xcolor, color, colortbl}
\usepackage{mdframed}
\usepackage{pifont}
\newcommand{\cmark}{\ding{51}}%
\newcommand{\xmark}{\ding{55}}%
\usepackage{enumitem}
\usepackage{graphicx} 
\usepackage{subfigure}
\usepackage{dsfont}
\usepackage{natbib}

\usepackage[toc,page]{appendix}
\usepackage{titlesec}

\usepackage[linesnumbered, ruled,vlined]{algorithm2e}
\SetKwBlock{Repeat}{repeat}{}

\usepackage{nicefrac}
\usepackage{booktabs}
\usepackage{multirow}
\usepackage{rotating}

\definecolor{mydarkblue}{rgb}{0,0.08,0.45}
\definecolor{mydarkgreen}{RGB}{39,130,67}
\newcommand{\green}{\color{mydarkgreen}}
\definecolor{mydarkred}{RGB}{192,47,25}
\newcommand{\red}{\color{mydarkred}}

\usepackage[colorlinks=true,
    linkcolor=mydarkblue,
    citecolor=mydarkblue,
    filecolor=mydarkblue,
    urlcolor=mydarkblue,
    pdfview=FitH]{hyperref}

\def\xx{{\boldsymbol x}}
\def\qq{{\boldsymbol q}}
\def\XX{{\boldsymbol X}}
\def\YY{{\boldsymbol Y}}
\def\ZZ{{\boldsymbol Z}}
\def\aa{{\boldsymbol a}}
\def\qq{{\boldsymbol q}}

\def\yy{{\boldsymbol y}}
\def\vv{{\boldsymbol v}}
\def\uu{{\boldsymbol u}}

\def\zz{{\boldsymbol z}}
\def\DD{{\boldsymbol D}}
\def\PP{{\boldsymbol P}}

\def\bphi{{\boldsymbol \phi}}
\def\QQ{{\boldsymbol Q}}
\def\HH{{\boldsymbol H}}
\def\balpha{{\boldsymbol \alpha}}

\def\RR{{\mathbb R}}
\def\EE{{\mathbb E}}
\newcommand{\Econd}{\mathbf{E}}
\renewcommand{\vec}{\mathbf{vec}}
\DeclareMathOperator{\prox}{\mathbf{prox}}
\def\defas{\stackrel{\text{def}}{=}}
\DeclareMathOperator*{\dom}{dom}

\DeclareMathOperator*{\Fix}{Fix}

\newcommand*\mybluebox[1]{%
\colorbox{myblue}{\hspace{1em}#1\hspace{1em}}}

\newcommand{\TOS}{\textsc{Tos}}
\newcommand{\SAGA}{\textsc{Saga}}
\newcommand{\VRTOS}{\textsc{Vr-Tos}}

\newcommand{\SVRG}{\textsc{Svrg}}
\newcommand{\ProxSVRG}{Prox\textsc{Svrg}}

\newcommand{\SGD}{\textsc{Sgd}}

\DeclareMathOperator*{\argmin}{{arg\,min}}
\DeclareMathOperator*{\minimize}{minimize}

\definecolor{myblue}{HTML}{D2E4FC}
\definecolor{Gray}{gray}{0.92}

\newmdtheoremenv{theo}{Theorem}
\newtheorem{theorem}{Theorem}
\newtheorem{corollary}{Corollary}
\newtheorem{lemma}{Lemma}
\newtheorem{definition}{Definition}

\newtheorem{innercustomgeneric}{\customgenericname}
\providecommand{\customgenericname}{}
\newcommand{\newcustomtheorem}[2]{%
  \newenvironment{#1}[1]
  {%
   \renewcommand\customgenericname{#2}%
   \renewcommand\theinnercustomgeneric{##1}%
   \innercustomgeneric
  }
  {\endinnercustomgeneric}
}

\newcustomtheorem{customtheorem}{Theorem}
\newcustomtheorem{customlemma}{Lemma}

\newcommand{\tablefont}[1] {{\fontsize{8}{10}\sffamily{#1}}}

\graphicspath{{./figures/}}

\begin{document}
\twocolumn[

\aistatstitle{
Proximal Splitting Meets Variance Reduction}

\aistatsauthor{ Fabian Pedregosa \And Kilian Fatras \And  Mattia Casotto }

\aistatsaddress{ ETH Z\"urich and UC Berkeley\footnotemark{} \\USA and Switzerland \And  Univ. Bretagne-Sud, \small{CNRS}, \small{IRISA}\\ \textsc{Inria} Rennes and \textsc{Obelix} \And Akur8 \\France} ]

\footnotetext[1]{Currently at Google AI, Canada}

\begin{abstract}
Despite the raise to fame of stochastic variance reduced methods like \SAGA\ and \ProxSVRG, their use in non-smooth optimization is still limited to a few simple cases. Existing methods require to compute the proximal operator of the non-smooth term at each iteration, which, for complex penalties like the total variation, overlapping group lasso or trend filtering, is an iterative process that becomes unfeasible for moderately large problems.
In this work we propose and analyze \VRTOS, a variance-reduced method to solve problems with an arbitrary number of non-smooth terms. Like other variance reduced methods, it only requires to evaluate one gradient per iteration and converges with a constant step size, and so is ideally suited for large scale applications.
Unlike existing variance reduced methods, it admits multiple non-smooth terms whose proximal operator only needs to be evaluated once per iteration.
We provide a convergence rate analysis for the proposed methods that achieves the same asymptotic rate as their full gradient variants and illustrate its computational advantage on 4 different large scale datasets.
\end{abstract}

\section{Introduction}

Stochastic variance reduced methods~\citep{roux2012stochastic,johnson2013accelerating, shalev2013stochastic} have been recently proposed as an improved alternative to the venerable stochastic gradient descent (\SGD) method~\citep{robbins1951}.
As \SGD, these methods only require to visit a small batch of random examples per iteration. This makes them ideally suited for large scale machine learning problems. Unlike \SGD, the variance of the updates decreases to zero --hence the name-- and converge with non-decreasing step sizes.

While initial stochastic variance reduced methods only considered smooth objectives, variants with support for a non-smooth term like \ProxSVRG~\citep{xiao2014proximal} and \SAGA~\citep{defazio2014saga} were soon developed. These methods are highly efficient whenever the nonsmooth part is \emph{proximal}, that is,  its proximal operator is available in closed form or at least fast to compute. This includes penalties such as the $\ell_1$ or group lasso norm, but not more complex ones like the overlapping group lasso~\citep{jacob2009group}, multidimensional total variation~\citep{barbero2014modular} or trend filtering~\citep{kim2009ell1}, to name a few.

A key observation is that many of these complex penalties can be
decomposed as a sum of proximal terms. Proximal splitting methods like the
three operator splitting \citep{davis2017three} or the Condat-V{\~u} algorithm~\citep{condat2013direct, vu2013splitting} then provide a principled
approach to incorporate these penalties into the optimization.
However, these methods require to compute the full gradient of the smooth term at each iteration, which can become costly in the context of large scale machine learning problems as it involves a full pass over the dataset. A question of key practical interest is whether these proximal splitting methods can be accelerated through the use of stochastic variance reduction techniques.

Our {\bfseries main contribution} is the development and analysis of \VRTOS, a stochastic variance reduced method that can solve problems with a sum of proximal terms.

The proposed method bridges two previously distant families of algorithms and inherit the best of both:
like the three operator splitting of \citet{davis2017three}, it can solve problems with multiple proximal terms, and like variance reduced stochastic methods its cost is independent on the number of smooth terms, converges with a fixed step size, and reaches the same asymptotic convergence rate than full gradient methods.
Furthermore, we also develop a sparse
variant of the proposed algorithm which can take advantage of
the sparsity in the input data.
The paper is organized as follows:
\begin{itemize}[leftmargin=*]
\item \emph{Method.} \S\ref{sec:methods} describes the \VRTOS\ algorithm, and extends it in \S\ref{scs:sparse} to leverage sparsity in the input data. \S\ref{scs:extension_k_terms} extends these methods to the case of an arbitrary number of proximal terms.

\item \emph{Analysis.} In \S\ref{scs:analysis} we provide a non-asymptotic convergence analysis of the proposed method. We show that, like other variance reduced methods, it converges with a fixed step size and can achieve the same asymptotic rate as the full gradient variants.

\item \emph{Experiments.} In \S\ref{scs:experiments} we compare the proposed method and related algorithms on a logistic regression problem with overlapping group lasso penalty on 4  datasets.
\end{itemize}

\subsection{Definitions and notation}
By convention, we denote vectors and vector-valued functions in lowercase boldface (e.g. $\xx$) and matrices in uppercase boldface letters (e.g. $\boldsymbol D$).
The { proximal operator } of a convex lower semicontinuous function $h$ is defined as
$  \prox_{\gamma h}(\xx) \defas \argmin_{\zz \in \RR^p} \{ h(\zz) + \frac{1}{2 \gamma}\|\xx - \zz\|^2\}$.
We say a function $f$ is $L$-smooth if it is differentiable and its gradient is $L$-Lipschitz, while it is { $\mu$-strongly convex} if $f - \frac{\mu}{2}\|\cdot\|^2$ is convex.

%
We denote by $[\,\xx\,]_b$ the $b$-th coordinate in $\xx$. This notation is overloaded so that for a collection of blocks $T = \{B_1, B_2, \ldots\}$, $[ \xx ]_T$ denotes the vector $\xx$ restricted to the coordinates in the blocks of $T$. For convenience, when $T$ consists of a single block $B$ we use $[ \xx ]_B$  as a shortcut of $[ \xx ]_{\{B\}}$.
Finally, we distinguish $\EE$, the full expectation taken with respect to all the randomness in the system, from $\Econd$, the conditional expectation with respect to the random index sampled at iteration $t$,  conditioned on all randomness up to iteration $t$.

\section{Methods}\label{sec:methods}

In this section we present our main contribution, the variance reduced three operator splitting method.
We will first consider problems with only two non-smooth terms, and generalize this formulation to an arbitrary number in \S\ref{scs:extension_k_terms}.

We consider the following optimization problem
\begin{empheq}[box=\mybluebox]{equation}\tag{OPT}\label{eq:opt_problem}
\begin{aligned}
  &\vphantom{\sum^i}\minimize_{\xx \in \RR^p}\, f(\xx) + g(\xx) + h(\xx) \,,\\
  & \vphantom{\sum_i^n}\text{ with } f(\xx) = \textstyle\frac{1}{n} \sum_{i=1}^n \psi_i(\xx) + \omega(\xx)~
  \end{aligned}
\end{empheq}
where each $\psi_i$ is convex and $L_\psi$-smooth, $\omega$ is convex and $L_\omega$-smooth and $g, h$ are \emph{proximal}, i.e., convex and we have access to their proximal operator.

This formulation allows to express a broad range of problems arising in machine learning and signal processing: the finite-sum includes common loss functions such as least squares or logistic loss; the two proximal terms $g, h$ can be extended to an arbitrary number  and include penalties such as the group lasso with overlap, total variation, $\ell_1$ trend filtering, etc. Furthermore, the proximal terms can be extended-valued, thus allowing for convex constraints through the use of the indicator function. With respect to previous work, this significantly enlarges the class of functions stochastic variance reduced methods can solve efficiently.

We allow the terms inside the finite sum to be an addition of two terms: $\psi_i$ and $\omega$. This might seem superfluous since it is not more general than the standard formulation with a single term. However, in practice $\psi_i$ (e.g., a least squares or logistic loss, see \ref{apx:implementation}) can be highly structured and allow for reduced storage schemes and/or have  sparse gradients (see \S\ref{scs:sparse}), properties which might not be shared by $\omega$, (e.g., an $\ell_2$ regularization term).

\begin{algorithm}[t]
 \KwIn{$\yy_0 \in \RR^p$, $\balpha_0 \in \RR^{n \times p}$, $\gamma > 0$}

 {\bfseries Temporary storage}: $\zz_t$, $\vv_t$ and $\xx_t$, all in $\RR^p$

 \KwResult{approximate solution to \eqref{eq:opt_problem} }

\For{$t=0, 1, \ldots $ }{

$\zz_t = \prox_{\gamma h}(\yy_t)$

Sample $i \in \{1, \ldots, n\}$ uniformly at random

$\vv_{t} = \nabla \psi_i(\zz_t) - \balpha_{i, t} + \overline{\balpha}_t + \nabla \omega(\zz_t)$

$\xx_t = \prox_{\gamma g}(2 \zz_t - \yy_t - \gamma \vv_{t})$

$\yy_{t+1} = \yy_t + \xx_t - \zz_t$

Update $\balpha_{t+1}$ according to \eqref{eq:q_memorization}\label{l:update_alpha}
}

\Return $\zz_t$

 \caption{Variance Reduced \TOS\  (\VRTOS)}\label{alg:vrtos}
\end{algorithm}

Central to our algorithm is the concept of $q$-memorization~\citep{hofmann2015variance}, which we recall below. It provides a convenient abstraction over common gradient memorization techniques like the ones in \SAGA\ and \SVRG.

\begin{definition}
A \emph{uniform $q$-memorization} algorithm selects at each iteration $t$ a random index set $J_t$ of memory terms to update according to
\begin{equation}\label{eq:q_memorization}
\balpha_{j, t+1} = \begin{cases}\nabla f_j(\zz_t) &\text{ if $j \in J_t$}\\\balpha_{j, t} &\text{ otherwise\,,}\end{cases}
\end{equation}
such that any $j$ has the same probability $q/n$ of being updated.
\end{definition}

We now introduce the variance-reduced three operator splitting (\VRTOS), a method to solve problems of the form \eqref{eq:opt_problem}. It is specified in Algorithm~\ref{alg:vrtos} and takes as input a vector of coefficients $\yy_0 \in \RR^p$, a table $\balpha_0 \in \RR^{n\times p}$ to store previous gradients and a step size $\gamma > 0$. Although in the general case this table is required to be of size $n\times p$, for linearly-parametrized loss functions like the logistic or least squares loss this can be reduced to size $n$ (\ref{apx:implementation}). Furthermore, the \SVRG-like update detailed below avoids the need for this storage at the expense of a lightly increased per iteration cost.

The proposed method performs one evaluation of each of the proximal terms and builds the gradient estimator $\vv_t$ from the table of previous gradients $\balpha_t$ and the index $i$ sampled uniformly at random. It is easy to see that $\vv_t$ is an unbiased estimate of the gradient, that is, $\Econd\,\vv_t = \nabla f(\zz_t)$.

This method allows the memory terms to be updated using any scheme that verifies the $q$-memorization framework (line \ref{l:update_alpha}). Some common schemes are:
\begin{itemize}[leftmargin=*]
\item \emph{\SAGA-like update}. At each iteration, the algorithm updates the same coefficient that has been sampled, i.e. $J_t = \{i\}$. In this scheme each memory term has probability $1/n$ of being updated, and so $q=1$.
\item \emph{\SVRG-like update}. Fix parameter $q > 0$ and draw at each iteration $r$ from a uniform distribution in the $[0, 1]$ interval. If $r < q/n$, the algorithm performs a complete update $\balpha_{j, t+1} = \nabla \psi_j(\zz_t)$ for all $j$, otherwise they are left unchanged.

Like in the \SVRG\ algorithm~\citep{johnson2013accelerating}, it is possible to avoid storing the memory terms since the $\overline\balpha_t$ is constant unless a full refresh is triggered. In this setting, only the $p$-dimensional vectors $\overline\balpha_t$ and $\widetilde{\zz}_t$ needs to be stored, where $\widetilde{\zz}_t$ is the value of $\zz_t$ last time a full refresh was triggered. This variant avoids the need to store $\balpha_t$, at the cost of a slight per iteration cost, as $\balpha_i = \nabla f_i(\widetilde{\zz}_t)$ needs to be computed at each iteration.

This memory update scheme was proposed by \citet{hofmann2015variance}, and unlike the original \SVRG\ algorithm the number of iterates between two full regresh is a random variable instead of a fixed number of iterations.
\end{itemize}

\subsection{Sparse \VRTOS}\label{scs:sparse}

\paragraph{Need for a sparse variant.} Modern web-scale optimization problems that arise in machine learning are not only large, they are also often \emph{sparse}.  For example, in the \texttt{LibSVM} datasets suite\footnote{\url{https://www.csie.ntu.edu.tw/~cjlin/libsvmtools/datasets/}}, 8 out of the 11 datasets with more than a million samples have a density below $0.01\%$, and the largest one in number of samples has a density below 1 per million.
Linearly-parametrized loss functions of the form $ \psi_i(\xx) = l_i(\aa_i^T \xx)$ have a gradient of the form $\nabla \psi_i(\xx) = \aa_i l_i'(\aa_i^T \xx)$, which inherits the same sparsity pattern as the data $\aa_i$. Since the data might be extremely sparse, it is hence of great practical interest to leverage sparsity in the partial gradients. This is the case in generalized linear models such as least squares or logistic regression, where $\aa_i$ are the rows of a data matrix.

In this subsection we assume that $g$ and $\omega$ are block separable, i.e.,
 can be decomposed block coordinate-wise as $g(\xx) = \textstyle{\sum_{B \in \mathcal{B}}} g_B([\xx]_B)$ and $\omega(\xx) = \textstyle{\sum_{B \in \mathcal{B}}} \omega_B([\xx]_B)$, where $\mathcal{B}$ is a partition of the coefficients into subsets which will call \emph{blocks} and $g_B, \omega_B$ only depends on coordinates in block $B$.
Furthermore, we will make use of the following notation:
\begin{itemize}[leftmargin=*]
\item \emph{Extended support}. We define the extended support of $\nabla \psi_i$, denoted $T_i$ as the set of blocks of $\mathcal{B}$ that intersect with its support, formally defined as $T_i \defas \{B: \text{supp}(\nabla f_i) \cap B \neq \varnothing, \,B\in\mathcal{B} \}$.
For totally separable penalties such as the $\ell_1$ norm, the blocks are individual coordinates and so the extended support covers the same coordinates as the support.

\item \emph{Reweighting constants}.
Let $\boldsymbol{P}_i$ be the projection onto the extended support, i.e., the diagonal matrix where $[\PP_i]_{B, B}$ is the identity if $B \in T_i$ and zero otherwise.
For simplicity we assume that each block appears in at least one $T_i$, as otherwise the problem can be reformulated without it.
For each block $B \in \mathcal{B}$ we define $d_B$ as the inverse frequency of that block in the extended support, i.e. $d_B = 1 / (\frac{1}{n}\sum_{i=1}^n \mathds{1}\{B \in T_i\})^{-1}$. For notational convenience we define the block-diagonal matrix $\DD$ as $[\DD]_{B, B} = d_B \boldsymbol{I}$ for each block $B \in \mathcal{B}$. Note that by definition $\frac{1}{n}\sum_{i=1}^n \PP_i = \DD^{-1}$. Computation of this diagonal matrix should be done as a preprocessing step of the algorithm.

\begin{algorithm}[t]
 \KwIn{$\yy_0 \in \RR^p$, $\balpha_0 \in \RR^{n \times p}$, $\gamma > 0$}

 {\bfseries Temporary storage}: $\zz_t$, $\vv_t$ and $\xx_t$, all in $\RR^p$

 \KwResult{approximate solution to \eqref{eq:opt_problem} }

\For{$t=0, 1, \ldots $ }{

Sample $i \in \{1, \ldots, n\}$ uniformly at random

$T_i = \text{extended support of } \nabla \psi_i$

$[\zz_t]_{T_i} = [\prox^{\DD^{-1}}_{\gamma h}(\yy_t)]_{T_i}$

$[\vv_{t}]_{T_i} = [\nabla \psi_i(\zz_t)\!-\!\balpha_{i, t} + \DD(\overline{\balpha}_t + \nabla \omega(\zz_t))]_{T_i}$

$[\xx_t]_{T_i} = [\prox_{\gamma \varphi_i}(2 \zz_t - \yy_t - \gamma \vv_{t})]_{T_i}$

$[\yy_{t+1}]_{T_i} = [\yy_t + \xx_t - \zz_t]_{T_i}$

Update $\balpha_{t+1}$ according to \eqref{eq:q_memorization}
}

\Return $\prox^{\DD^{-1}}_{\gamma h}(\yy_t)$

 \caption{Sparse \VRTOS}\label{alg:vrtos_sparse}
\end{algorithm}

\item The \emph{scaled proximal operator} is defined for a function $\varphi$, step size $\gamma > 0$, positive definite matrix $\boldsymbol{H}$ and norm $\|\cdot\|_{\HH}^2 \defas \langle \cdot, \HH\cdot \rangle$ as
\begin{equation}
  \prox^{\boldsymbol{H}}_{\gamma \varphi}(\xx) \defas \argmin_{\zz \in \RR^p}\big\{\, \varphi(\zz) + \frac{1}{2\gamma}\|\xx - \zz\|_{\boldsymbol{H}}^2\,\big\}
\end{equation}
\end{itemize}

We now have all necessary ingredients to present the sparse variant of \VRTOS. This is specified in Algorithm~\ref{alg:vrtos_sparse}. In this variant, all operations are restricted to the extended support.

The algorithm requires to compute the scaled proximal operators of $g$ and $h$. By block separability of $g$ its scaled proximal operator can be computed in block-wise as $[\prox^{\DD^{-1}}_{\gamma g}(\xx)]_{B} = [\prox_{(d_B \gamma) h}(\xx)]_B$ for all $B \in \mathcal{B}$.
Hence the cost of computing $[\xx_t]_{T_i}$ will depend on the extended support size and not on the dimensionality.

We can unfortunately not guarantee the same complexity for $[\zz_t]_{T_i}$ since we do not have a closed form for the scaled proximal operator of $h$ in general.
We review some specific cases in which it is possible to compute this scaled proximal operator in  \ref{apx:optimizing_multiple_penalties}.
Alternatively, in the next subsection we propose a reformulation that avoids the need to compute this scaled proximal operator at the expense of higher memory usage.

In the case that one proximal term is zero, the proposed algorithm with \SAGA-like update of the memory terms defaults to the Sparse \SAGA\ variant of \citet{pedregosa2017breaking}. With \SVRG-like update of the memory terms it instead yields a novel sparse variant of Prox\SVRG~\citep{xiao2014proximal}.
For both of the proposed algorithms, when input is dense, $\PP_i = \DD = \boldsymbol{I}$ and we recover Algorithm~\ref{alg:vrtos}.

\subsection{Extension to an arbitrary number of proximal terms}\label{scs:extension_k_terms}

The proposed method can be easily extended to the more general setting of an objective function with an arbitrary number of proximal terms of the form
\begin{empheq}[box=\mybluebox]{equation}\tag{OPT-$k$}\label{eq:obj_fun_k}
\begin{aligned}
    &\minimize_{\xx \in \RR^p}\, f(\xx) + \textstyle\sum_{j=1}^k g_j(\xx)\,,\nonumber \\
    &\text{ with } f(\xx) = \textstyle\frac{1}{n} \sum_{i=1}^n \psi_i(\xx) + \omega(\xx)~,
\end{aligned}
\end{empheq}
where $\psi_i$ and $\omega$ are as in \eqref{eq:opt_problem} and  $g_1, \ldots, g_k$ are proximal.
This is done by expressing the above as a problem of the form~\eqref{eq:opt_problem} in an enlarged space and then applying the proposed algorithm to this reformulation. For this, we will introduce $k$ new variables which we will constrain to be equal via an indicator function. The above problem can be written equivalently as follows,
\begin{equation*}\label{eq:obj_extended}
\min_{\XX \in \RR^{k\times p}}\,f(\overline\XX) + \underbrace{\textstyle\sum_{j=1}^k g_j(\XX_j)}_{\defas g(\XX)} + \underbrace{\imath\{\XX_1\!=\!\cdots\!=\!\XX_k\}}_{\defas h(\XX)}\,,
\end{equation*}
where we have split the original variable into $k$ variables $\XX_1, \ldots, \XX_k$ and constrained them to be equal using an indicator function in the last term.
In this formulation the first term is smooth, and the other two terms are proximal. The second term is proximal since the variables in $g_i$ are decoupled, each $g_i$ is proximal by assumption and the last term is an indicator function over a linear subspace, and hence its scaled proximal operator can be computed in closed form as follows (Lemma \ref{lemma:projection_kterms}):
\begin{align}
    &[\prox^{\boldsymbol{D}^{-1}}_{\gamma h}(\XX)]_{i, j} = \left({\textstyle\sum_{i=1}^n} a_{i, j} \XX_{i, j}\right) / \left({\textstyle\sum_{i=1}^n} a_{i, j}\right)\nonumber \\
    &\text{ with } a_{i, j} = \DD^{-1}_{i p + j, i p + j}~,
\end{align}

Hence, the problem with multiple proximal terms \eqref{eq:obj_fun_k} can be formulated as a problem with two proximal terms \eqref{eq:opt_problem} and so it is possible to apply the proposed method defined in the previous subsections.
This gives a variance reduced method  for problems with an arbitrary number of proximal term.
It is worth noting that for the sparse variants this formulation avoids the potentially difficult computation of the scaled proximal operator of $h$. 



\section{Related work}\label{scs:related_work}

{ \begin{table*}[t]
\centering
\footnotesize
\setlength\tabcolsep{5pt}\begin{tabular}{c c | c c c c |}
\cline{2-6}
\multicolumn{1}{c|}{} & \multirow{2}{*}{Methods} &
\multirow{1}{*}{incremental} & \multirow{1}{*}{non-decreasing}  & multiple non-smooth & \multirow{2}{*}{sparse updates}\\
\multicolumn{1}{c|}{} &  &\multirow{-1}{*}{updates} & \multirow{-1}{*}{step size}  & terms & \\
\cline{2-6}

\multicolumn{1}{c|}{} & {\cellcolor{Gray} \VRTOS  } & \cellcolor{Gray} & \cellcolor{Gray} & \cellcolor{Gray} & \cellcolor{Gray}   \\
\multicolumn{1}{c|}{\multirow{4}{*}}&\footnotesize(\emph{this work})
\cellcolor{Gray}&
\multirow{-2}{*}{\green\large\cmark} \cellcolor{Gray}
\cellcolor{Gray}&
\cellcolor{Gray} \multirow{-2}{*}{\green\large\cmark}  & \cellcolor{Gray} \multirow{-2}{*}{\green\large\cmark}& \cellcolor{Gray} \multirow{-2}{*}{\green\large\cmark}\\

\multicolumn{1}{c|}{}&{ \SAGA  } &
 & & &
\\
\multicolumn{1}{c|}{\multirow{4}{*}}&
\multirow{-1}{*}{\citep{defazio2014saga}}&
\multirow{-2}{*}{\green\large\cmark}
&
 \multirow{-2}{*}{\green\large\cmark} &\multirow{-2}{*}{\red\large\xmark} & \multirow{-2}{*}{{\green{\large\cmark}}\scriptsize\citep{pedregosa2017breaking}}\\

 \multicolumn{1}{c|}{\multirow{4}{*}}&\multirow{1}{*}{\cellcolor{Gray} \ProxSVRG  } &
 \cellcolor{Gray} &\cellcolor{Gray} &\cellcolor{Gray} &\cellcolor{Gray} \\
 \multicolumn{1}{c|}{\multirow{4}{*}}&
 \multirow{-1}{*}{\citep{xiao2014proximal}}\cellcolor{Gray}&
 \multirow{-2}{*}{\green{\large\cmark}} \cellcolor{Gray}
 \cellcolor{Gray}&
 \cellcolor{Gray} \multirow{-2}{*}{\green{\large\cmark}} &  \multirow{-2}{*}{\red\large\xmark} \cellcolor{Gray} & \multirow{-2}{*}{{\red\large\xmark}{\,${\dagger}$}}  \cellcolor{Gray}\\

\multicolumn{1}{c|}{\multirow{4}{*}}&\multirow{1}{*}{ \TOS  } &
 & & & \\
\multicolumn{1}{c|}{\multirow{4}{*}}&
\multirow{-1}{*}{\citep{davis2017three}}&
\multirow{-2}{*}{\red\large\xmark}
&
 \multirow{-2}{*}{\green{\large\cmark}} &  \multirow{-2}{*}{\green{\large\cmark}}& \multirow{-2}{*}{N/A}\\

 \multicolumn{1}{c|}{\multirow{4}{*}}&\multirow{1}{*}{\cellcolor{Gray} Stochastic \TOS  } &
 \cellcolor{Gray} &\cellcolor{Gray} & \cellcolor{Gray} & \cellcolor{Gray} \\
 \multicolumn{1}{c|}{\multirow{4}{*}}&
 \multirow{-1}{*}{\citep{yurtsever2016stochastic}}\cellcolor{Gray}&
 \multirow{-2}{*}{} \cellcolor{Gray}
 \cellcolor{Gray}\multirow{-2}{*}{\large\green\cmark}&
 \cellcolor{Gray} \multirow{-2}{*}{{\red\large\xmark}} & \cellcolor{Gray}\multirow{-2}{*}{\large\green\cmark} & \cellcolor{Gray} \multirow{-2}{*}{{\red\large\xmark}}\\

\cline{2-6}\vspace{-2.8ex}\\
\end{tabular}
\caption{{\bfseries Comparison with related work.} The proposed method is unique in that it combines the advantages of variance-reduced methods (incremental updates, non-decreasing step sizes and sparse updates) with the advantages of proximal splitting (support for multiple non-smooth terms). $\dagger$: a sparse variant of \ProxSVRG\ follows as a special case of Algorithm~\ref{alg:vrtos_sparse} with $h=0$ and the SVRG-like update of the memory terms.}\label{table:related_work}
\end{table*}
}

We comment on the most closely related ideas, summarized in Table \ref{table:related_work}.

Methods that support objective functions of the form \eqref{eq:opt_problem} with two or more proximal terms and a smooth term accessed via its gradient have recently been proposed. Examples are the
the primal-dual hybrid gradient method (also known as the Condat-V{\~u})~\citep{condat2013primal,vu2013splitting},\footnote{We note that this method  can optimize the more general objective function $f(\xx) + g(\xx) + h(\boldsymbol{L}\xx)$, for an arbitrary linear operator $\boldsymbol{L}$ that is fixed to the identity in our setting.}
the generalized forward-backward splitting~\citep{raguet2013generalized} or the three operator splitting~\citep{davis2017three}. Due to its excellent empirical performance and amenability to sparse updates we have chosen this last method as the basis for the proposed method. The proposed \VRTOS\ method can be seen as a generalization of this last method, as both method are identical when $n=1$.

A different stochastic variant of the three operator splitting was proposed by \citet{yurtsever2016stochastic} for the slightly more general case in which $f$ is given by an expectation. Like the proposed algorithms, this method only needs to evaluate the gradient of one element in the finite sum per iteration.
Unlike the proposed methods, the variance of the updates does not decrease to zero and requires --as other non-variance reduced method-- a decreasing step size. Furthermore, all updates are dense even in the presence of sparse gradients so the method performs poorly on large sparse problems.

\citep{balamurugan2016stochastic} proposed a variance-reduced method to solve problems a general class of saddle point problems including
$\min_\xx\max_\uu \frac{1}{n}\sum_{i=1}^n f_i(\xx) + M(\xx, \uu)$,
where $M(\cdot)$ is proximal. With $M(\xx, \uu) = g(\xx) + \langle \xx, \uu\rangle - h^*(\uu)$, this is equivalent to the problem in \eqref{eq:opt_problem}. However, the method requires $M$ to be strongly concave in $\uu$, which is equivalent to $h$ being smooth, and so is not applicable to the same class of problems as the proposed method. We note that this requirement is not merely an artifact of the theory, as the algorithm requires knowledge of this smoothness parameter.

Stochastic variance-reduced variants of ADMM have also been recently proposed, see e.g. \citep{zheng2016stochastic, yu2017fast}. Compared to the proposed methods, none of the existing variants support sparse updates and require tuning more than one step-size parameter.

\section{Analysis}\label{scs:analysis}

In this section we provide a non-asymptotic convergence rate analysis for the proposed method:
\vspace{-0.5em}\begin{itemize}[leftmargin=*]
\item All the proposed variants converge with a step size $1/(3L_f)$, with $L_f \defas L_\psi + d_{\max}L_\omega$, where $d_{\max}$ is the maximum element in the diagonal matrix $\DD$ ($d_{\max} = 1$ for non-sparse variants).
\item For \VRTOS\ (Algorithm~\ref{alg:vrtos}) we obtain convergence rates that asymptotically match those of the full-gradient variant, i.e., $\mathcal{O}(1/t)$ convergence rate for convex problems (Theorem~\ref{thm:sublinear_convergence}) and a linear convergence rate under strong convexity of $f$ and smoothness of $h$ (Theorem~\ref{thm:linear_convergence}).
\item For the sparse variant, Sparse \VRTOS\ (Algorithm~\ref{alg:vrtos_sparse}), we obtain a linear convergence rate under the same assumptions (Theorem~\ref{thm:linear_convergence}). However, for general convex objectives we could only obtain a worse $\mathcal{O}(1/\sqrt{t})$ convergence rate (Theorem~\ref{thm:sublinear_convergence_sparse}).
\end{itemize}

In this section we will use the following {\bfseries extra notation}.
We define the following primal ($\mathcal{P}$), and dual function ($\mathcal{D}$) as:
\begin{align}
    &\mathcal{P}(\xx) \defas f(\xx) + g(\xx) + h(\xx)~,\nonumber \\
    &\mathcal{D}(\uu) \defas (f + g)^*(-\uu) + h^*(\uu)~,\label{eq:primal_dual_loss}
\end{align}
where $^*$ denotes the Fenchel conjugate. We denote by $\xx^\star$ an arbitrary minimizer of the primal objective and define the ``dual iterate'' $\uu_t \defas \DD^{-1}(\yy_t - \zz_t) / \gamma$ ($\DD = \boldsymbol{I}$ for the dense variants).
We also define the following generalized three operator splitting operator:
\begin{align}\label{eq:operator}
    &\boldsymbol{G}_{\gamma}(\yy) \defas \yy  -  \zz_\yy + \prox^{\DD^{-1}}_{\gamma g}(2\zz_\yy - \yy - \gamma \DD\nabla f(\zz_\yy))\,,\nonumber \\
    &~\text{with $\zz_\yy = \prox^{\DD^{-1}}_{\gamma h}(\yy)$}~,
\end{align}

and its set of fixed points, which we denote  $\Fix(\boldsymbol{G}_\gamma)$.
Another quantity that will appear often in the analysis is $H_0 \defas {1}/{(2  n L_f)}\sum_{i=1}^n \|\balpha_{i, 0} - \psi_{i}(\xx^\star)\|^2$.

Throughout this section we make the following two technical assumptions:
\vspace{-0.5em}\paragraph{Assumption 1: Regularity.} We assume each $\psi_i$ is $L_\psi$-smooth, $\omega$ is $L_{\omega}$-smooth, $g$ and $h$ are proper (i.e., have nonempty domain), lower semicontinuous (i.e., its sublevel sets are closed) convex functions. We recall that lower semicontinuity is a weak form of continuity that allows extended-valued functions with domain over a closed set.

{\bfseries Assumption 2: Qualification conditions.} We assume the relative interior of $\dom g$ and $\dom h$ have a non-empty intersection.
This is a very weak and standard assumption, which allows to rule out pathological cases such as disjoint domains and allows to relate the primal and dual optimal objective (see e.g.\citep[Proposition 15.13]{bauschke2017convex} or \citep[Proposition 5.3.8]{bertsekas2015convex}), a property sometimes referred to as strong or total duality.

{\bfseries Sublinear convergence}.
The following theorem shows a $\mathcal{O}(1/t)$ convergence rate for \VRTOS\ on arbitrary convex objectives.

One of the issues when analyzing the convergence of the three operator splitting is that the objective function might be $+\infty$, for example when both proximal terms are an indicator function.
Following \citet{chambolle2016ergodic,pedregosa2018adaptive}, we will state the convergence rate for general functions in terms of the \emph{saddle point suboptimality}, defined as
\begin{align}
    &\mathcal{L}(\widetilde\xx, \uu) - \mathcal{L}(\xx, \widetilde\uu)\,,~\text{ with }~\nonumber \\
    &\mathcal{L}(\xx, \uu) \defas f(\xx) + g(\xx)+  \langle \xx, \uu\rangle - h^*(\uu)~,
\end{align}
where $\mathcal{L}$ is the Lagrangian associated with $\mathcal{P}$ and $\mathcal{D}$. As \citet{davis2017three}, we will also state convergence rates in terms of the objective suboptimality under a Lipschitz assumption on $h$ in \eqref{eq:sub_copnvergence_obj}.

\begin{theorem}\label{thm:sublinear_convergence}
Let $\overline{\xx}_t$ denote the averaged (also known as ergodic) iterate, i.e., $\overline{\xx}_t = {(\sum_{k=0}^t\xx_k )/(t+1)}$ and $\overline{\uu}_t = (\sum_{k=0}^t\uu_k )/(t+1)$. Then the \VRTOS\ method (Algorithm~\ref{alg:vrtos}) converges for any step size $\gamma \leq 1 / (3L_f)$, and for $\gamma = 1/ (3L_f)$ we have the following bound for all $(\xx, \uu) \in \dom g \times \dom h^*$:
\begin{equation}\label{eq:sub_copnvergence_saddle}
\EE\left[\mathcal{L}(\overline\xx_t, \uu) - \mathcal{L}(\xx, \overline\uu_t)\right] \leq \frac{10 n }{q(t+1)}C_0,
\end{equation}
with $\yy = \xx + \gamma \uu$, $\yy^\star \in \Fix(\boldsymbol{G}_\gamma)$, and $C_0 = \left[\frac{3 L_{f}q}{20 n}\|\yy_0 - \yy\|^2 + \frac{3 L_{f}q}{2n}\|\yy_0 - \yy^\star\|^2 + H_0  \right]~$, where we recall $H_0 = {1}/{(2  n L_f)}\sum_{i=1}^n \|\balpha_{i, 0} - \psi_{i}(\xx^\star)\|^2$.

Furthermore, if $h$ is $\beta_h$-Lipschitz we have the following rate in terms of the primal objective:
\begin{equation}\label{eq:sub_copnvergence_obj}
    \mathcal{P}(\overline{\xx}_t) - \mathcal{P}(\xx^\star) \leq \frac{10 n }{q(t+1)}\widetilde{C}_0 ~,
\end{equation}
with $\widetilde{C}_0 = \frac{6 L_{f}q}{20 n}\|\zz_0 - \xx^\star\|^2 + \frac{3 L_{f}q}{2n}\|\yy_0 - \yy^\star\|^2 + \frac{q}{15 n L_f}\beta_h^2 + H_0$.
\end{theorem}
The previous theorem gives a $\mathcal{O}(1/t)$ convergence rate in terms of the saddle point suboptimality for arbitrary convex functions and $\mathcal{O}(1/t)$ rate in function suboptimality under a Lipschitz assumption on $h$, matching the strongest bounds of \SAGA~\citep{defazio2014saga}.

For their sparse variants, however, we have only been able to prove a slower $\mathcal{O}(1/\sqrt{t})$ rate on the operator residual, despite the fact that in practice the algorithm exhibits a much faster empirical convergence (see \S \ref{scs:experiments}). \ref{apx:fixed_point_characterization} contains a characterization of the fixed points of this operator that justifies why this is a meaningful suboptimality criterion for \eqref{eq:opt_problem}. Although there is no direct correspondence between rates on the gradient and on objective values, lower bounds are asymptotically equivalent~\citep{nesterov2012make}.
\newcommand{\TheoremSubSparse}{
Sparse \VRTOS\ (Algorithm~\ref{alg:vrtos_sparse}) converges for every step size ${\gamma \leq 1/(3L_f)}$. In particular, for $\gamma = {1}/{(3L_f)}$ and $\yy_t$ obtained after $t \geq 1$ updates we have the bound
  \begin{equation}
\min_{k=0, \ldots, t}\left\{ \EE \|\yy_k - \boldsymbol{G}_\gamma(\yy_k)\| \right\} \leq \sqrt{ \frac{C_0}{L q (t+1)}}  = \mathcal{O}\left(\frac{1}{\sqrt{t}}\right),
\end{equation}
with $C_0= \frac{5 d_{\max}n}{{Lq(t+1)}} \left[({2Lq}/{n})\|\yy_0 - \yy^\star\|^2 + H_0\right]$.
}
\begin{theorem}\label{thm:sublinear_convergence_sparse}
\TheoremSubSparse
\end{theorem}

{\bfseries Linear convergence.}
The three operator splitting has been shown to have a linear convergence rate under the assumption of strong convexity of the smooth term and smoothness of one of the proximal terms~\citep[\S4.4]{davis2015three}. Although this last condition is rarely verified in practice since its main application is on non-smooth proximal terms, it is instructive to see that the proposed method --despite the reduced cost per iteration-- also enjoys a linear convergence rate under the same assumptions.

\newcommand{\TheoremLinear}{
Let $\psi_i$ be $\mu_\psi$-strongly convex and $\omega$ be $\mu_{\omega}$-strongly convex, where $\mu_\psi + \mu_\omega > 0$. Furthermore, let $h$ be $L_h$-smooth.
Then for any step size $\gamma \leq {1}/{(3 L_f)}$, all the proposed methods converge geometrically in expectation. For $\gamma = {1}/{(3 L_f)}$, we have the following bound for Algorithm \ref{alg:vrtos} ($d_{\max{}}=1$ in this case) and Algorithm~\ref{alg:vrtos_sparse}:
  \begin{equation}
\EE \|\zz_{t+1} - \xx^\star\|^2 \leq \left(1 - \min \Big \{ \frac{q}{4n}, \frac{1}{3 d^3_{\max}\delta^2 \kappa} \Big \}\right)^t  D_0\quad,
\end{equation}
with $D_0 \defas {d_{\max}}\left[\frac{q}{2\gamma(1 - \gamma\mu)n} \|\yy_0 - \yy^\star\|^2 + H_0\right]$,  $\delta = (1 + L_h/(3L_f))$, $\kappa = L_f/\mu$ and $\yy^\star \in \Fix(\boldsymbol{G}_\gamma)$.
}
\begin{theorem}[Linear convergence]\label{thm:linear_convergence}
\TheoremLinear
\end{theorem}
%

{ \begin{table*}[t]
\centering
\footnotesize
\hspace*{-1.4cm}
\setlength\tabcolsep{5pt}\begin{tabular}{c c | c c c c |}
\cline{2-6}
\multicolumn{1}{c|}{} & {Method} &
{step size} &  Proximal oracle & Convergence rate & {Extra assumptions}\\
\cline{2-6}

\multicolumn{1}{c|}{\multirow{2}{*}{\begin{sideways}\textbf{\sffamily Geometric}\end{sideways}}}
&
{\cellcolor{Gray} \SAGA  } &
\cellcolor{Gray} &
 \cellcolor{Gray} &
\cellcolor{Gray} &
\cellcolor{Gray} \\
\multicolumn{1}{c|}{\multirow{4}{*}}&\scriptsize\citep{defazio2014saga}
\cellcolor{Gray}&
\multirow{-2}{*}{\Large\nicefrac{$1$\,}{\,$3L_f$}} \cellcolor{Gray} &
\multirow{-2}{*}{$\mathbf{prox}_{\gamma(g + h)}$}
\cellcolor{Gray}&
\cellcolor{Gray} \multirow{-2}{*}{$\displaystyle \Big( 1 - \min\big\{\textstyle\frac{1}{4 n},\frac{1}{3 \kappa}  \big\} \Big)^t C_0$} &
 \cellcolor{Gray} \multirow{-2}{*}{{Each $\psi_i$ is  $\mu$-cvx}} \\

\multicolumn{1}{c|}{}&{ \ProxSVRG  } &
 &
 &
 &
\multirow{2}{*}{$f$ is $\mu$-cvx}\\
\multicolumn{1}{c|}{\multirow{4}{*}}&
\multirow{-1}{*}{\scriptsize\citep{xiao2014proximal}}&
\multirow{-2}{*}{\Large\nicefrac{$1$\,}{\,$10L_f$}}  &
\multirow{-2}{*}{$\mathbf{prox}_{\gamma(g + h)}$}
&
 \multirow{-2}{*}{$\Big(\frac{1}{\kappa 0.6 m } + \frac{2}{3}\Big)^t  C_0$} &\\

\multicolumn{1}{c|}{} &
{\cellcolor{Gray} \VRTOS  } &
\cellcolor{Gray} &
 \cellcolor{Gray} &
\cellcolor{Gray} &
\cellcolor{Gray} \multirow{-1}{*}{{Each $\psi_i$ is  $\mu$-cvx}}\\
\multicolumn{1}{c|}{\multirow{4}{*}}&{\scriptsize(this work)}
\cellcolor{Gray}&
\multirow{-2}{*}{\Large\nicefrac{$1$\,}{\,$3L_f$}} \cellcolor{Gray} &
\multirow{-2}{*}{$\mathbf{prox}_{\gamma g}~,~\mathbf{prox}_{\gamma h}$}
\cellcolor{Gray}&
\cellcolor{Gray} \multirow{-2}{*}{$\displaystyle \Big( 1 - \min\big\{\textstyle\frac{q}{4 n},\frac{1}{3 d_{\max{}}^2\delta^2 \kappa}  \big\} \Big)^t C_0$} &
 \cellcolor{Gray} \multirow{-1}{*}{{and $h$ is $L_h$-smooth}} \\
 \cline{2-6}\addlinespace[0.2cm]

\cmidrule{2-6}\vspace{-3.2ex}\\
\cline{2-6}
\multicolumn{1}{c|}{\multirow{4}{*}}& {\cellcolor{Gray}\multirow{-1}{*}{ \SAGA}  } &
\cellcolor{Gray} &
  \cellcolor{Gray} &
\cellcolor{Gray} &
\cellcolor{Gray}  \\
\multicolumn{1}{c|}{\multirow{4}{*}}&
\multirow{-1}{*}{\scriptsize\citep{defazio2014saga}}\cellcolor{Gray}&
\multirow{-2}{*}{\Large\nicefrac{$1$\,}{\,$3L_f$}} \cellcolor{Gray} &
\multirow{-2}{*}{$\mathbf{prox}_{\gamma (g + h)}$}
\cellcolor{Gray}& \multirow{-2}{*}{$\mathcal{O}(1/t)$}
\cellcolor{Gray} \multirow{-2}{*}{} &
 \cellcolor{Gray} \multirow{-2}{*}{None}\\
\multicolumn{1}{c|}{\multirow{-2}{*}{\begin{sideways}\textbf{\sffamily Sublinear}\end{sideways}}}
& {\multirow{-1}{*}{ Stochastic \textsc{Tos}}  } &
 &
 &
 &
\multirow{-1}{*}{$f$ is $\mu$-cvx +} \\
\multicolumn{1}{c|}{\multirow{4}{*}}&
\multirow{-1}{*}{\scriptsize \citep{yurtsever2016stochastic}}&
\multirow{-2}{*}{$\mathcal{O}${\Large(\nicefrac{1\,}{\,$t$})}} &
\multirow{-2}{*}{$\mathbf{prox}_{\gamma g}~,~\mathbf{prox}_{\gamma h}$}
&
\multirow{-2}{*}{$\mathcal{O}(1/t)$} &
\multirow{-1}{*}{bound on gradients} \\
\multicolumn{1}{c|}{\multirow{4}{*}} &
\cellcolor{Gray}\VRTOS &
\cellcolor{Gray}&
\cellcolor{Gray}&
\cellcolor{Gray}\multirow{3}{*}{$\displaystyle $} &  \cellcolor{Gray}\\
\multicolumn{1}{c|}{\multirow{4}{*}}& \cellcolor{Gray}{\scriptsize (this work, dense/sparse variant)}
&\cellcolor{Gray} \multirow{-2}{*}{\Large\nicefrac{$1$\,}{\,$3L_f$}}
&\cellcolor{Gray} \multirow{-2}{*}{$\mathbf{prox}_{\gamma g}~,~\mathbf{prox}_{\gamma h}$}
&\cellcolor{Gray} \multirow{-2}{*}{$\mathcal{O}(1/{t})$ / $\mathcal{O}(1/\sqrt{t})$}
&\cellcolor{Gray}  \multirow{-2}{*}{None} \\
\cline{2-6}
\end{tabular}
\caption{{\bfseries Assumptions and properties of related incremental methods.} In every case, we take the step size recommended by the theory, where we assume $\omega=0$ to make them comparable. Proximal oracle is the proximal operators that are needed by the algorithm. Extra assumptions refer to those other than Assumptions 1 and 2.  The linear rates use the quantities $\delta = (1 + \gamma L_h)$,  $\kappa = L_f/\mu$.  For \ProxSVRG, $m$ denotes the epoch size and the convergence rate is relative to the number of epochs and not iterations like the rest.
}
\label{table:convergence_rates}
\end{table*}
}

\subsection{Discussion}\label{scs:discussion}

{\bfseries Comparison of convergence rates}. We summarize the obtained convergence rates for the proposed methods and compare them against the best known rates for related stochastic methods in Table~\ref{table:convergence_rates}.
In the linearly-convergent regime, we obtain rates that are similar to \SAGA\, but with the rate factor multiplied by ${1}/{(\delta^2 d^3_{\max})}$, quantity that depends on the smoothness of $g$ and the sparsity of the gradients.

{\bfseries An improved \ProxSVRG\ variant}. The analysis of \ProxSVRG~\citep{xiao2014proximal} requires that the step size verifies an implicit equation that depends among other things on the strong convexity parameter. For typical choices of the parameters this is $1/(10L_f)$~\citep[Theorem 1]{xiao2014proximal}. In contrast, Sparse \VRTOS\ with \SVRG-like sampling with $h=0$ yields a variant of \ProxSVRG\ with more favorable properties. First, none of its parameters depend on the strong convexity constant (while still obtaining a linear convergence rate since $L_h=0$ in this case), which is most often unknown. Second, it admits the much larger step size  $1/(3L_f)$, which is, to the best of our knowledge, the largest step size of any \SVRG\ variant. Third, it can leverage sparsity in the input data through sparse updates.

{\bfseries Linear convergence without smoothness of the proximal term}. Theorem \ref{thm:linear_convergence} requires smoothness of one of the proximal terms to guarantee linear convergence. Despite this, linear convergence is observed in practice without this assumption (Figure~\ref{fig:bench_fused}). This has also been observed in the case of the original (non-variance reduced) three operator splitting~\citep{davis2017three, pedregosa2018adaptive}, although an explanation for this is still an open problem.
Furthermore, the lack of linear convergence when both proximal terms are non-smooth does not seem to be a limitation of the proof, as a counterexample was provided in  \citep[Appendix D.6]{davis2015three}. In this work, the authors constructed a strongly monotone operator with a sublinear convergence.

%

{\bfseries Step size adaptivity to linear convergence}. A practical consequence of the above theorems is that using the same step size $\gamma={1}/{(3 L_f)}$ we obtain a sublinear convergence by Theorem \ref{thm:sublinear_convergence} and a linear rate (under additional assumptions) by Theorem \ref{thm:linear_convergence}. That is, one can use the ``universal'' step size ${1}/{(3L_f)}$ and automatically obtain linear convergence whenever the assumptions of Theorem \ref{thm:linear_convergence} are verified.

{\bfseries Limitations}. The following are some scenarios under which the proposed method is
expected to perform poorly. The cost in computation and storage scales linearly with the number of proximal terms, hence it cannot cope with other scenarios with many nonsmooth terms such as empirical risk minimization with the hinge loss or group lasso with overlap with a large number of overlaps (for instance $> 100$).
Also, there are still penalties that cannot be reduced to a sum of proximal terms, such as the nuclear norm. Algorithms based on Frank-Wolfe~\citep{jaggi2013revisiting} or with approximate proximal operators ~\citep{schmidt2011convergence} might be better suited in such regimes.

\section{Experiments}\label{scs:experiments}

\begin{figure*}[t]
\centering
\begin{tabular}{lrrrr}
\toprule
{\bfseries\sffamily Dataset} & \multicolumn{1}{c}{\#\tablefont{samples}} & \multicolumn{1}{c}{\#\tablefont{dimensions}
} & {\tablefont{density}} & \multicolumn{1}{c}{$L_f/\mu$}\\
\midrule
{\bfseries\sffamily RCV1 (full)}~\citep{lewis2004rcv1} & \hfill 697,641 & \hfill 47,236 & \hfill $1.5 \times 10^{-3}$ & \hfill 2.50 $\times 10^4$ \\
{\bfseries\sffamily URL}~\citep{ma2009identifying} & \hfill 2,396,130 & \hfill 3,231,961 & \hfill $3.5 \times 10^{-5}$ & \hfill 1.28 $\times 10^7$\\
{\bfseries\sffamily KDD10}~\citep{yu2010feature} & \hfill 19,264,097 & \hfill 29,890,095 & \hfill 9.8 $\times 10^{-7}$ & \hfill $5.2\times10^8$ \\
{\bfseries\sffamily Criteo}~\citep{juan2016field} & \hfill 45,840,617 & \hfill 1,000,000 & \hfill $3.8\times 10^{-5}$ & \hfill $1.1\times 10^7$\\
\bottomrule
\end{tabular}

\vspace*{1em}
\centering \includegraphics[width=0.92\linewidth]{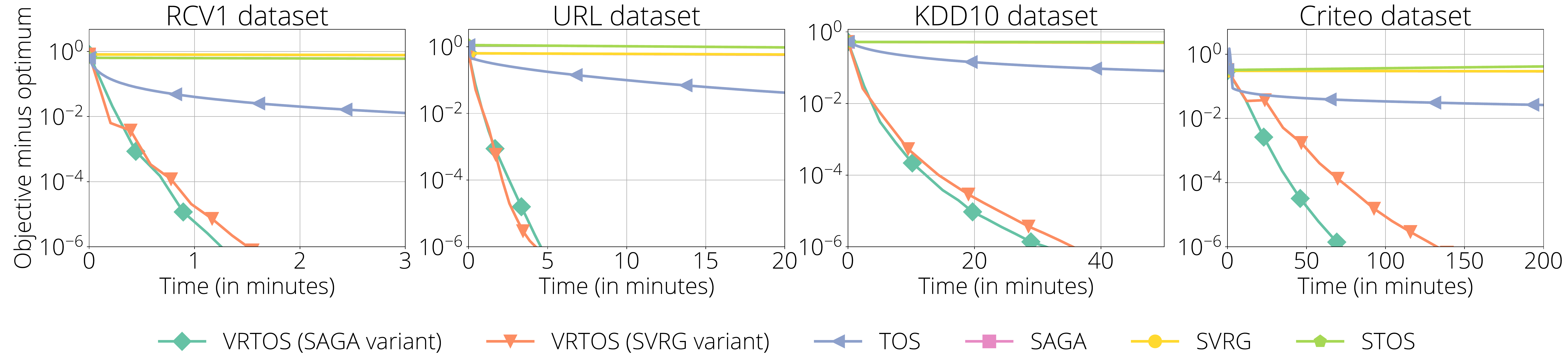}
\caption{{\bfseries Top}: Description of considered datasets. {\bfseries Bottom:} Suboptimality vs time of different algorithms on a logistic regression with overlapping group lasso penalty problem. }
\label{fig:bench_fused}
\end{figure*}

Although the proposed methods can be applied more broadly, we consider for the experiments a logistic regression problem with squared $\ell_2$ regulrization and an overlapping group lasso penalty~\citep{jacob2009group}. Following \citet{jacob2009group} we choose
groups  of 10 variables with 2 variables of overlap between two successive groups: $\{\{1, \ldots, 10\}, \{8, \ldots, 18\}, \{16, \ldots, 26\}, \ldots\}$.
The amount of group regularization was chosen such that the solution has roughly $10\%$ of non-zero coefficients and the of $\ell_2$ regularization was fixed to ${1}/{n}$. We consider the following methods:
\begin{itemize}[leftmargin=*]
\item The proposed method Sparse \VRTOS\ (Algorithm~\ref{alg:vrtos_sparse}), where the overlapping group lasso penalty is split as a sum of two non-overlapping group lasso penalties, for which the proximal operator is available in closed form. We used the formulation with 3 proximal terms of \S\ref{scs:extension_k_terms} to better leverage sparsity in the dataset and consider \SAGA\ and \SVRG-like updates, denoted \VRTOS\ (\SAGA\ variant) and \VRTOS\ (\SVRG\ variant) respectively. This implementation is publicly available in the C-OPT package.\footnote{\url{http://openopt.github.io/copt/}}

It is worth noting that while original penalty is \emph{not} block separable, each of the terms in the splitting as two group lasso penalties \emph{is} block separable. This will allow us to make a much more efficient use of sparsity than what is possible on on methods like \SAGA\ and \ProxSVRG.

\item The three operator splitting (denoted \TOS), in its recently proposed variant with adaptive step size~\citep{pedregosa2018adaptive}.
\item The stochastic three operator spitting of \citep{yurtsever2016stochastic} with the same splitting as \VRTOS, denoted \textsc{Stos}.
\item \SAGA\ and \ProxSVRG, where the proximal operator is evaluated approximately using 10 iterations of the Douglas-Rachford method.
\end{itemize}

The above methods were compared on 4 large-scale datasets described in the table of Figure~\ref{fig:bench_fused}. Further details and implementation aspects are discussed in \ref{apx:implementation}.

The best performing algorithms overall are the proposed \VRTOS\ variants, which are over an order of magnitude faster than the second best method, the adaptive three operator splitting. The stochastic three operator splitting, not being able to take advantage of the sparsity in the gradients, performs poorly in this benchmark, appearing as a straight line. \SAGA\ and \ProxSVRG\ were the slowest since they require to compute a costly proximal operator at each iteration and are unable to leverage the sparsity of the dataset due to the non-block-separability of the non-smooth term.

It is worth noting from Figure~\ref{fig:bench_fused} that the two variants of Sparse \VRTOS\ exhibit an empirical linear convergence, despite the fact that the theory only predicts in this regime a much slower $\mathcal{O}(1/\sqrt{t})$ convergence rate (Theorem~\ref{thm:sublinear_convergence}).

We provide extra experiments in \ref{apx:overlapping_benchmarks}.

\section{Future work}


This work can be extended in several ways. As highlighted in \S\ref{scs:discussion}, a theoretical explanation for the empirical linear convergence without smoothess of any proximal term, even for the full gradient algorithm, is lacking. We conjecture \emph{partly smooth} is a sufficient condition on the penalties to ensure local linear convergence, as recently proven for related methods~\citep{liang2018local}.
Second, we conjecture that the convergence rate of the sparse variant can be improved up to to  $\mathcal{O}(1/t)$.
A third direction for future work would be the development an extension that allow for a linear operator inside one of the proximal terms, as in \citep{condat2013direct,zhao2018stochastic,yan2018new}.

\section*{Acknowledgements}
The authors warmly thank
Vincent Roulet, Vlad Niculae, R\'emi Leblond and Federico Vaggi for their feedback on the manuscript, as well as Adrien Taylor, Alexandre D'Aspremont, Gabriel Peyr\'e, Guillaume Obozinski, P. Balamurugan, Francis Bach and Marwa El Halabi for fruitful discussions.

This work has been done while FP was under funding from the European Union's Horizon 2020 research and innovation programme under the Marie Sklodorowska-Curie grant agreement 748900. KF is funded through the project OATMIL ANR-17-CE23-0012 of the French National Research Agency (ANR).
Computing time on was donated by Amazon through the program ``AWS Cloud Credits for Research''.

\bibliographystyle{apalike}
\bibliography{biblio}

\clearpage
\appendix
\onecolumn

\titleformat{\section}{\Large\bfseries}{\thesection}{1em}{}
\titleformat{\subsection}{\large\bfseries}{\thesubsection}{1em}{}
\gdef\thesection{Appendix \Alph{section}}



{\centering{\LARGE\bfseries Proximal Splitting Meets Variance Reduction}

\vspace{1em}
\centering{{\LARGE\bfseries Supplementary material}}

}

\paragraph{Outline.} The supplementary material of this paper is organized as follows.
\begin{itemize}[leftmargin=*]
    \item \ref{apx:basic_defs} presents basic definitions and properties that will be used throughout the proofs but which are not specific to our methods. Most of these can be found in convex optimization textbooks, such as \citep{bauschke2017convex, nesterov2004introductory}.
    \item \ref{apx:fixed_point_characterization} give a characterization of the fixed points of the three operator splitting, relating the set of fixed points of the three operator splitting to the solutions of primal and dual objectives. This is a stronger result than the one stated in \citep{davis2017three} and used in some of our proofs.
    \item \ref{apx:proofs} gives the proofs of those results in the Analysis section of the paper.
    \item \ref{apx:optimizing_multiple_penalties} discusses splitting strategies for different penalties and examines some cases in which the scaled proximal operator can be computed in closed form.
    \item \ref{apx:experiments} discusses implementation aspects of the proposed algorithms.
\end{itemize}

\vspace{1em}

\section{Basic definitions and properties}\label{apx:basic_defs}

\vskip 1em

\begin{definition}[proper function] A function $f: \mathcal{X} \subseteq \RR^p \to ]-\infty, \infty]$ is said to be proper if its domain is not empty.
\end{definition}

\vskip 0.5em

\begin{definition}[Fenchel conjugate] The Fenchel conjugate of a function $f: \mathcal{X} \subseteq \RR^p \to ]-\infty, \infty]$ is defined as
\begin{equation}
f^*(\xx^\star) = \sup \left \{ \left. \left\langle \xx^\star , \xx \right\rangle - f \left( \xx \right) \right| \xx \in \mathcal{X} \right\}.
\end{equation}
\end{definition}

\vskip 0.5em

\begin{definition}[lower semicontinuity] We say that a proper convex function $f$ is lower-semicontinuous if all of its levelsets ${\{\xx \in \dom(f)\,|\,f(\xx) \leq \alpha \}}$ are closed.
\end{definition}

\vskip 0.5em

\begin{definition}[relative interior] The relative interior of a convex set $C \subseteq \RR^p$ is defined as
\begin{equation}
\operatorname{relint}(C) \defas \{\xx \in C : \forall {\yy \in C} \; \exists {\lambda > 1}: \lambda \xx + (1-\lambda)\yy \in C\}
\end{equation}
\end{definition}

\vskip 0.5em

\begin{definition}[Bregman divergence]\label{def:bregman}
 The Bregman divergence associated with a convex function $f$ for points $\xx, \yy$ in its domain is defined as:
  $$
  B_f(\xx, \yy) \defas f(\xx) - f(\yy) - \langle \nabla f(\yy), \xx - \yy \rangle
  $$
  Note that this is always positive due to the convexity of $f$.
\end{definition}

\vskip 0.5em

\begin{definition}[Proximal operators]
Here, we redefine 2 variants of a critical notion. The proximal operator is defined for a function $\varphi$, step size $\gamma > 0$ as:
\begin{equation}
  {\prox_{\gamma \varphi}}(\xx) \defas {\argmin_{\zz \in \RR^p}}\big\{\, \varphi(\zz) + \frac{1}{2\gamma}\|\xx - \zz\|\big\}\,
\end{equation}
The \emph{scaled proximal opeartor} is defined for a function $\varphi$, step size $\gamma > 0$ and positive definite matrix $\boldsymbol{H}$ as the solution of the following optimization problem
\begin{equation}
  {\prox^{\boldsymbol{H}}_{\gamma \varphi}}(\xx) \defas {\argmin_{\zz \in \RR^p}}\big\{\, \varphi(\zz) + \frac{1}{2\gamma}\|\xx - \zz\|_{\boldsymbol{H}}^2\,\big\}\,~\text{ with $\|\cdot\|_{\HH}^2 \defas \langle \cdot, \HH\cdot \rangle$}.
\end{equation}
\end{definition}

\vskip 0.5em

\begin{lemma}[subgradient characterization of proximal operator]\label{lemma:prox_characterization}
Let $g$ be a convex proper lower semicontinuous function. Then for any $
\xx$, $\HH$ positive definite matrix and any $\gamma > 0$ we have the following characterization of proximal operator:
\begin{equation}
\zz = \prox^{\HH}_{\gamma g}(\xx) \iff \HH(\xx - \zz)/{\gamma} \in \partial g(\zz)
\end{equation}
\end{lemma}
\begin{proof}
By the definition of proximal operator we have that $\zz = \prox_{\gamma g}^{\HH}(\xx)$ is equivalent to
\begin{align}
&\zz \in \argmin_{\zz' \in \RR^p} g(\zz') + \frac{1}{2\gamma}\| \zz' - \xx\|_{\HH}^2\\
&\iff 0 \in \partial g(\zz) + \frac{\HH}{\gamma}(\zz - \xx)\\
&\iff \frac{\HH}{\gamma}(\xx - \zz) \in \partial g(\zz)
\end{align}
where the first equivalence is a consequence of the first order optimality conditions.
\end{proof}

\begin{lemma}[Conjugate-inverse identity]\label{lemma:conjugate_inverse}
Let $h$ be a convex, proper lower semicontinuous function. Then
\begin{equation}
\uu \in \partial h(\zz) \iff \zz \in \partial h^*(\uu)~.
\end{equation}
In other words, $(\partial h)^{-1} = \partial h^*$.
\end{lemma}
\begin{proof}
See e.g. \citep[Corollary 16.30]{bauschke2017convex} or \citep[Proposition 11.3]{rockafellar1998variational}.
\end{proof}

\begin{lemma}[Generalized variance decomposition]\label{lemma:generalized_variance_decomposition}
Let $\zeta_i$ be a random variable and let $\Econd $ the expectation with respect to this random variable. Furthermore, let $\QQ_i$ be an orthogonal projection such that $\QQ_i \boldsymbol{\zeta}_i = \boldsymbol{\zeta}_i$, $\Econd \QQ_i$ is invertible and $\boldsymbol{A} = (\Econd \QQ_i)^{-1}$. Then we have
	\begin{equation}
	\Econd \|\boldsymbol\zeta_i - \QQ_i\boldsymbol{A} \Econd \boldsymbol\zeta_i\|^2 = \Econd \|\boldsymbol\zeta_i\|^2 - \|\Econd \boldsymbol\zeta_i\|^2_{\boldsymbol{A}}~.\label{eq:variance_decomp2}
	\end{equation}
\end{lemma}
\begin{proof}
The assumption of $\QQ_i$ being an orthogonal projection implies that it is symmetric and idempotent. Developing the square we have
\begin{align}
	\Econd \|\boldsymbol\zeta_i - \QQ_i \boldsymbol{A} \Econd \boldsymbol\zeta_i\|^2 &= \Econd \|\boldsymbol\zeta_i\|^2 + \Econd\langle\QQ_i \boldsymbol{A}  \Econd\boldsymbol\zeta_i, \QQ_i \boldsymbol{A}  \Econd\boldsymbol\zeta_i\rangle - 2 \Econd \langle \boldsymbol\zeta_i, \QQ_i\boldsymbol{A}  \Econd \boldsymbol\zeta_i\rangle \\
    &= \Econd \|\boldsymbol\zeta_i\|^2 + \Econd\langle\QQ_i \QQ_i \boldsymbol{A}  \Econd\boldsymbol\zeta_i,  \boldsymbol{A}  \Econd\boldsymbol\zeta_i\rangle - 2 \Econd \langle \boldsymbol\QQ_i\zeta_i, \boldsymbol{A}  \Econd \boldsymbol\zeta_i\rangle \\
    &\qquad \text{ (by symmetry of $\QQ_i$)}\nonumber\\
	&= \Econd \|\boldsymbol\zeta_i\|^2 + \Econd\langle\QQ_i \boldsymbol{A}  \Econd\boldsymbol\zeta_i, \boldsymbol{A} \Econd\boldsymbol\zeta_i\rangle - 2 \Econd \langle \boldsymbol\zeta_i,  \boldsymbol{A} \Econd \boldsymbol\zeta_i\rangle \\
	&\qquad \text{ (idempotence of $\QQ_i$ and assumption $\QQ_i \boldsymbol\zeta_i = \boldsymbol\zeta_i$ respectively) }\nonumber\\
	&= \Econd \|\boldsymbol\zeta_i\|^2 + \langle  \Econd\boldsymbol\zeta_i, \boldsymbol{A} \Econd\boldsymbol\zeta_i\rangle - 2  \langle \Econd\boldsymbol\zeta_i,  \boldsymbol{A}\Econd \boldsymbol\zeta_i\rangle \\
	&\qquad \text{ (taking expectations )}\nonumber\\
	& = \Econd \|\boldsymbol\zeta_i\|^2 - \|\Econd \boldsymbol\zeta_i\|^2_{\boldsymbol{A}}
\end{align}
\end{proof}


\hfill

\begin{lemma}[Smooth inequality 1]\label{lemma:l_smooth_ineq}
  Let $f_i$ be $L_f$-smooth and convex for $i=1, \ldots, n$. Then it is verified that
  \begin{equation}
  \frac{1}{n}\sum_{i=1}^n\|\nabla f_i(\xx) - \nabla f_i(\yy)\|^2 \leq 2 L_{f}B_f(\xx, \yy)~.
  \end{equation}
\end{lemma}

\begin{proof}
Since each $f_i$ is $L_f$-smooth, it is verified \citep[Theorem 2.1.5]{nesterov2004introductory} that
\begin{equation}
\|\nabla f_i(\xx) - \nabla f_i(\yy)\|^2 \leq 2 L_{f}(f_i(\xx) - f_i(\yy) - \langle \nabla f_i(\yy), \xx - \yy\rangle)
\end{equation}
The result is obtained by averaging over $i$.
\end{proof}

\vspace{0.5em}

\begin{lemma}[Bound on matrix norm]\label{lemma:bound_matrix_norm}
Let $\DD$ be a diagonal matrix with strictly positive diagonal elements,  let $d_{\max}$ and $d_{\min}$ denote its maximum and minimum diagonal entry respectively. Then for any $\xx \in \RR^p$ we have the following inequalities
\begin{equation}
d_{\min{}}\|\xx\|^2\leq \|\xx\|^2_{\DD} \leq d_{\max}\|\xx\|^2
\end{equation}
\end{lemma}
\begin{proof}
By definition of the $\DD$-norm we have
\begin{gather}
\|\xx\|^2_{\DD} = \sum_{i=1}^n \DD_{i, i}{\xx_i^2} \leq \sum_{i=1}^n d_\text{max}{\xx_i^2} = d_\text{max}\|\xx\|^2\\
\|\xx\|^2_{\DD} = \sum_{i=1}^n \DD_{i, i}{\xx_i^2} \geq \sum_{i=1}^n d_\text{min} {\xx_i^2} = d_\text{min}\|\xx\|^2
\end{gather}
The result follows from chaining both inequalities
\end{proof}

\vspace{0.5em}

\begin{lemma}[Properties of proximal operator]\label{lemma:properties_prox}
  Let $g$ be a convex lower semicontinuous function and $\HH$ a symmetric positive definite matrix. Then for all $\yy, \widetilde{\yy}$ we have the following inequality, often referred to as firm nonexpansiveness:
  \begin{equation}\label{eq:nonexpansive}
       \|\prox^{\HH^{-1}}_{\gamma g}(\yy) - \prox^{\HH^{-1}}_{\gamma g}(\widetilde{\yy})\|_{\HH}^2 \leq \langle \prox^{\HH^{-1}}_{\gamma g}(\yy) - \prox^{\HH^{-1}}_{\gamma g}(\widetilde{\yy}), \yy - \widetilde{\yy}\rangle_{\HH} ~
  \end{equation}
  Furthermore, if $g$ is $L_g$-smooth and $\HH$ has smallest singular value $\sigma_{\min}$ and largest singular value $\sigma_{\max}$, then we also have the following bound:
  \begin{equation}\label{eq:prox_l_smooth}
    \| \yy - \widetilde{\yy}\| \leq (\sigma_{\max}/\sigma_{\min} + \gamma L \sigma_{\max})\|\prox^{\HH^{-1}}_{\gamma g}(\yy) - \prox^{\HH^{-1}}_{\gamma g}(\widetilde{\yy})\|
  \end{equation}
\end{lemma}

\begin{proof}
\underline{\emph{First inequality}}.
Let $\zz \defas \prox^{\HH^{-1}}_{\gamma g}(\yy)$, $\widetilde\zz \defas\prox^{\HH^{-1}}_{\gamma g}(\widetilde{\yy})$. By the subgradient characterization of Lemma~\ref{lemma:prox_characterization} we have
\begin{equation}
\HH^{-1} (\yy - \zz) / \gamma \in \partial g(\zz)\quad,\quad\HH^{-1} (\widetilde\yy - \widetilde\zz) / \gamma \in \partial g(\widetilde\zz)
\end{equation}
Since the subdifferential of a convex function is monotonous, in particular $\partial g$ is monotonous, and so we have
\begin{align}
\langle \HH^{-1} (\yy - \zz) / \gamma - \HH^{-1} (\widetilde\yy - \widetilde\zz) / \gamma, \zz - \widetilde\zz \rangle \geq 0 &\iff \langle  (\yy - \zz) - (\widetilde\yy - \widetilde\zz) , \zz - \widetilde\zz \rangle_{\HH^{-1}}\!\geq 0\\
&\iff \langle\yy- \widetilde\yy, \zz - \widetilde\zz \rangle_{\HH^{-1}}  \geq \|\zz - \widetilde\zz\| _{\HH^{-1}}
\end{align}
which proves the first part of the lemma (firm nonexpansive).

  \hfill

  \underline{\emph{Second inequality}}.
 To prove the second inequality we will use a generalization of the argument from~\citet[Proposition 1]{giselsson2016linear}. Let $g_{\gamma}$ be defined as
  \begin{equation}
  g_{\gamma} \defas \gamma g + \frac{1}{2}\|\cdot\|_{\HH^{-1}}^2\quad.
  \end{equation}
  By the subgradient characterization of the proximal operator (Lemma~\ref{lemma:prox_characterization}) and the conjugate-inverse identity (Lemma~\ref{lemma:conjugate_inverse}), we have
  \begin{align}
      &\zz = \prox^{\HH^{-1}}_{\gamma g}(\yy)\\
      &\Longleftrightarrow \zz \in \{\xx\, |\, \HH^{-1}(\yy - \xx) \in \gamma\nabla g(\xx)\}\qquad \text{(Lemma~\ref{lemma:prox_characterization})}\\
      &{\Longleftrightarrow} \zz \in \{\xx\, |\, \HH^{-1}\yy \in \nabla g_\gamma(\xx)\}\\
      &\Longleftrightarrow \zz \in (\nabla g_\gamma)^{-1}(\HH^{-1}\yy)\\
      &\Longleftrightarrow \zz = \nabla g^*_\gamma(\HH^{-1}\yy) \quad \text{(Lemma~\ref{lemma:conjugate_inverse})}\label{eq:equiv_prox_grad}
  \end{align}
  where $g^*_{\gamma}$ denotes the convex conjugate of $g_{\gamma}$. Note that we can write $\nabla$ in the last term instead of $\partial$ for $g^*_{\gamma}$ because this function is $1$-smooth with respect to the $\HH^{-1}$-norm by the strong convexity of $g_\gamma$.

The term $\frac{1}{2}\|\cdot\|_{\HH^{-1}}^2$ is $\sigma^{-1}_{\min}$-smooth and so $g_{\gamma}$ is $(\sigma^{-1}_{\min} + \gamma L_g)$-smooth. By the duality between Lipschitz gradient and strong convexity (see e.g., \citet{rockafellar1998variational}), $g_\gamma^*$ is $1 / (\sigma^{-1}_{\min} + \gamma L)$-strongly convex. Then for  arbitrary $\yy$ and $\widetilde\yy$ we have
  \begin{align}
  &\| \nabla g^*_\gamma(\HH^{-1}\yy) - \nabla g^*_\gamma(\HH^{-1}\widetilde{\yy}) \| \| \HH^{-1}\yy - \HH^{-1}\widetilde{\yy}\|\\
  &\qquad\geq
  \langle \nabla g^*_\gamma(\HH^{-1}\yy) - \nabla g^*_\gamma(\HH^{-1}\widetilde{\yy}), \HH^{-1}\yy - \HH^{-1}\widetilde{\yy} \rangle\qquad \text{ (by Cauchy-Schwarz)}\\
  &\qquad\geq \frac{1}{1/\sigma_{\min} + \gamma L_g }\|\HH^{-1}\yy - \HH^{-1}\widetilde{\yy}\|^2  \qquad \text{(by strong convexity)}
\end{align}
The result is trivial when $\yy = \widetilde{\yy}$. We can then assume $\yy \neq \widetilde{\yy}$,
and dividing both sides by ${\|\HH^{-1}\yy - \HH^{-1}\widetilde{\yy}\|}$ (non-zero by assumption) we obtain
\begin{align}
\| \nabla g^*_\gamma(\HH^{-1}\yy) - \nabla g^*_\gamma(\HH^{-1}\widetilde{\yy}) \| &\geq \frac{1}{1/\sigma_{\min} + \gamma L_g }\|\HH^{-1}\yy - \HH^{-1}\widetilde{\yy}\|\\
&\geq \frac{\sigma^{-1}_{\max}}{1/\sigma_{\min} + \gamma L_g }\|\yy - \widetilde{\yy}\|~.
\end{align}
where the last inequality we have used that $\yy \to \frac{1}{2}\yy^T \HH^{-1} \yy$ is strongly convex with strong convexity parameter $\sigma_{\max}^{-1}$ and $\HH^{-1}\yy$ is the gradient of this function.

Using now the equivalence between $\nabla g^*_\gamma(\HH^{-1}\cdot)$ of Eq. \eqref{eq:equiv_prox_grad} and the proximal operator we finally have the claimed bound:
  \begin{align}
\| \yy - \widetilde{\yy}\| &\leq {(\sigma_{\max}/\sigma_{\min} + \gamma L_g \sigma_{\max} })\| \nabla g^*_\gamma(\HH^{-1}\yy) - \nabla g^*_\gamma(\HH^{-1}\widetilde{\yy}) \| \\
&= {(\sigma_{\max}/\sigma_{\min} + \gamma L_g \sigma_{\max})}\|\prox_{\gamma g}(\yy) - \prox_{\gamma g}(\widetilde{\yy})\|~.
  \end{align}

\end{proof}

\begin{lemma}[Block firm non-expansiveness]\label{lemmma:block_nonexpansive}
  Let $\xx, \widetilde{\xx}$ be two arbitrary vectors in $\RR^p$, $g$ be a block-separable convex lower semicontinuous function with blocks $\mathcal{B}$. Let $\zz \defas \prox_{\gamma g}(\xx)$, $\widetilde{\zz} \defas \prox_{\gamma g}(\widetilde{\xx})$. Then for any subset $\mathcal{A}\subseteq \mathcal{B}$ it is verified that:
\begin{equation}
\langle [\zz - \widetilde{\zz}]_{\mathcal{A}}, [\xx - \widetilde{\xx}]_{\mathcal{A}} \rangle \geq \|[\zz - \widetilde{\zz}]_{\mathcal{A}}\|^2\,.
\end{equation}
\end{lemma}
\begin{proof}
  By the block-separability of $g$, the proximal operator is the concatenation of the proximal operators of the blocks. In other words, for any block $A \in\mathcal{A}$ we have:
\begin{equation}
  [\zz]_A = \prox_{\gamma g_A}([\xx]_A)~,\quad [\widetilde{\zz}]_A = \prox_{\gamma g_A}([\widetilde{\xx}]_A)\,,
\end{equation}
where $g_A$ is the restriction of $g$ to $A$.
  By firm non-expansiveness of the proximal operator (see e.g.~\citet[Proposition 4.2]{bauschke2017convex}) we have that:
  $$
  \langle [\zz]_A - [\widetilde{\zz}]_A, [\xx]_A - [\widetilde{\xx}]_A\rangle \geq \|[\zz]_A - [\widetilde{\zz}]_A\|^2 \,.
  $$
Summing over the blocks in $\mathcal{A}$ yields the desired result.
\end{proof}

\clearpage

\section{Fixed point characterization}\label{apx:fixed_point_characterization}

In this subsection we provide a characterization of the fixed points of the three operator splitting.

The following theorem characterizes the set of fixed points $\boldsymbol{G}_{\gamma}$ (defined in \eqref{eq:operator}) as the weighted Minkowski sum of primal and dual solutions. We will denote by $\Fix(\boldsymbol{G}_{\gamma})$ the set of fixed points of $\boldsymbol{G}_{\gamma}$.
This characterization seems to be new, and it will be used in some of the later proofs.

\newcommand{\theoremonetext}{Let $\mathcal{P}^\star$ denote the set of minimizers of the primal objective and $\mathcal{D}^\star$ the set of minimizers of the dual objective. Then the set of fixed points of the three splitting is
\begin{equation}\label{eq:fix_decomposition}
\Fix(\boldsymbol{G}_\gamma) = \mathcal{P}^\star + \gamma\, \DD\mathcal{D}^\star = \left\{\xx + \gamma \DD \uu\,|\,\xx \in \mathcal{P}^\star, \uu \in \mathcal{D}^\star\right\}~.
\end{equation}}
\begin{theorem}[Fixed point for operator splitting]\label{thm:fixed_point} \theoremonetext
\end{theorem}

\begin{proof}
We  first characterize the fixed points of $\boldsymbol{G}_\gamma$ by a subdifferential inclusion. Given $\yy\in \RR^p$, let $\zz \defas \prox^{\DD^{-1}}_{\gamma h}(\yy)$ and $\uu \defas \DD^{-1}(\yy - \zz)/\gamma$. Consider the following sequence of equivalences:
\begin{align}
  \yy = \boldsymbol{G}_\gamma(\yy)
    &\iff \begin{cases}\zz = \prox^{\DD^{-1}}_{\gamma g}(2\zz - \yy - \gamma\DD \nabla f(\zz))\\
    \zz = \prox^{\DD^{-1}}_{\gamma h}(\yy)
    \end{cases}
    &\text{ (by definition of $\boldsymbol{G}_\gamma$)}\label{eq:y_ty_0}\\
  &\iff \begin{cases}\DD^{-1}(-\frac{\gamma}{\gamma}(\yy - \zz) - \gamma \DD\nabla f(\zz))/{\gamma} \in \partial g(\zz) \label{eq:t13}\\
  \DD^{-1}(\yy - \zz)/{\gamma} \in \partial h(\zz)\end{cases}&\text{ (by Lemma~\ref{lemma:prox_characterization})}\\
  &\iff \begin{cases} -\uu \in  \partial (f + g)(\zz)\\
  \uu \in \partial h(\zz)
  \end{cases} \label{eq:t14}\\
  &\iff \begin{cases} \zz \in  \partial (f + g)^*(-\uu)\\
  \zz \in \partial h^*(\uu )
  \end{cases}& \text{(by Lemma~\ref{lemma:conjugate_inverse})}\label{eq:t15}
\end{align}

The rest of the proof is divided in two parts, proving in the first part that $\Fix(\boldsymbol{G}_\gamma) \subseteq {\mathcal{P}^\star + \gamma \DD \mathcal{D}^\star}\,$, and the reverse inclusion in the second part.

\emph{Part 1}. Our goal is to prove $\Fix(\boldsymbol{G}_\gamma) \subseteq {\mathcal{P}^\star + \gamma \mathcal{D}^\star}\,$. Let $\yy \in\Fix(\boldsymbol{G}_\gamma)$ and $\zz$, $\uu$ be as defined above. From their definition we immediately have $\zz + \gamma \DD\uu = \yy$, and so we only need to prove that $\zz, \uu$ are minimizers of the primal and dual objective respectively.
By definition of $\zz$ we have the following subdifferential inclusions
\begin{align}
\zz = \prox^{\DD^{-1}}_{\gamma h}(\yy) &\iff
\frac{\DD^{-1}}{\gamma}(\yy - \zz) = \uu \in \partial h(\zz)\label{eq:t11}\\
&\iff \zz \in \partial h
^*(\uu)\label{eq:t12}~,
\end{align}
where we have used Lemma~\ref{lemma:prox_characterization} for the first equivalence and Lemma~\ref{lemma:conjugate_inverse} for the second one.
Adding together \eqref{eq:t11} with the first line of \eqref{eq:t14}, and \eqref{eq:t12} minus the first line of \eqref{eq:t15} gives
\begin{align}
&\qquad \,0 \in   \partial h(\zz) + \partial g(\zz) + \nabla f(\zz)\label{eq:first_order_optimality_primal}\\
& \text{ and }~ 0 \in  \partial h^*(\uu) - \partial(f + g)^*(-\uu)~,\label{eq:first_order_optimality_dual}
\end{align}
and so by the first-order optimality conditions $\zz$ and $\uu$ are minimizers of the primal and dual objectives respectively. We have proved $\Fix(\boldsymbol{G}_\gamma) \subseteq \mathcal{P}^\star + \gamma \DD \mathcal{D}^\star\,$.

\vspace{0.5em}\emph{Part 2}. Our goal now is to prove the inverse inclusion, ${\mathcal{P}^\star + \gamma \DD \mathcal{D}^\star} \subseteq \Fix(\boldsymbol{G}_\gamma)$. Let $(\xx, \uu) \in \mathcal{P}^\star \times \mathcal{D}^\star$, we will prove that $\yy \defas \xx + \gamma \DD\uu$ is a fixed point of $\Fix(\boldsymbol{G}_\gamma)$.

We start by recalling the notion of \emph{paramonotinicity}, which will play a key role in this part of the proof. This notion was introduced by \citet{Iusem1998} and is key to characterizing the set of fixed points of related methods, such as the Douglas-Rachford splitting~\citep{bauschke2012attouch}. An operator $\boldsymbol{C}$ is said to be paramonotonic if the following implication is verified
\begin{equation}
\begin{rcases*}
\boldsymbol{a}^\star \in \boldsymbol{C} \boldsymbol{a}\\
\boldsymbol{b}^\star \in \boldsymbol{C} \boldsymbol{b}\\
\langle \boldsymbol{a}^\star - \boldsymbol{b}^\star, \boldsymbol{a} - \boldsymbol{b}\rangle = 0
\end{rcases*} \implies \boldsymbol{a}^\star \in \boldsymbol{C} \boldsymbol{b} \text{ and } \boldsymbol{b}^\star \in \boldsymbol{C} \boldsymbol{a}~.
\end{equation}

The usefulness of this notion in this case comes from the fact that the subdifferential of a convex proper lower semicontinuous function is paramonotonic~\citep[Proposition 2.2]{Iusem1998}. Hence we have that $\partial h$ and $\partial (f + g)$ are paramonotonic.

By the first-order optimality conditions on the primal and dual loss we have that there exists elements $\uu_\zz$ and $\zz_\uu$ such that
\begin{gather}
\uu_\zz \in \partial h(\zz) \cap (-\partial (f + g)(\zz))\label{eq:u_z_inclusion}\\
\zz_\uu \in \partial h^*(\uu) \cap (\partial (f + g)^*(-\uu))~,
\end{gather}
where the second inclusion can be written equivalently using the conjugate-inverse identity (Lemma~\ref{lemma:conjugate_inverse}) as
\begin{equation}
\uu \in \partial h(\zz_\uu) \cap (- \partial (f + g)(\zz_\uu))~.\label{eq:u_inclusion}
\end{equation}
Using Eq.~\eqref{eq:u_z_inclusion} and \eqref{eq:u_inclusion} we have by monotony of $\partial h$ and $\partial (f+g)$
\begin{equation}
\langle \uu_\zz - \uu, \zz - \zz_\uu \rangle\geq 0 ~\text{ and } \langle \uu_\zz - \uu, \zz - \zz_\uu \rangle\leq 0
\end{equation}
from where we necessarily have $\langle \uu_\zz - \uu, \zz - \zz_\uu \rangle = 0$. We hence have by paramonotonicity of $\partial h$
\begin{equation}
\begin{rcases}
\uu_\zz \in \partial h(\zz)\\
\uu \in \partial h(\zz_\uu)\\
\langle \uu_\zz - \uu, \zz - \zz_\uu\rangle = 0
\end{rcases} \implies \uu \in \partial h(\zz)
\end{equation}
Similarly, by paramonotonicity of $\partial (f+g)$ we have
\begin{equation}
\begin{rcases}
-\uu_\zz \in \partial (f+g)(\zz)\\
-\uu \in \partial (f+g)(\zz_\uu)\\
\langle \uu_\zz - \uu, \zz - \zz_\uu\rangle = 0
\end{rcases} \implies -\uu \in \partial (f + g)(\zz)
\end{equation}
Combining the last two equations we have by the definition of $\yy$ the following inclusions
\begin{equation}
\begin{cases}-\uu \in \partial (f+g)(\zz)\\
\uu \in \partial h(\zz)\end{cases}
\end{equation}
which by Eq.~\eqref{eq:t14} implies that $\yy \in \Fix(\boldsymbol{G}_\gamma)$ (note that these are all equivalences from \eqref{eq:y_ty_0} to \eqref{eq:t14}). This concludes the proof.
\end{proof}

\begin{corollary}[Minimizer of our objective]\label{cor:fixed_point}
Let $\yy \in \Fix(\boldsymbol{G}_\gamma)$.
Then we have that $\zz = \prox^{\DD^{-1}}_{\gamma h}(\yy)$ is a minimizer of the primal objective $\mathcal{P}$, $\prox^{\DD^{-1}}_{\gamma h}(\yy) = \prox^{\DD^{-1}}_{\gamma h}(2 \zz - \yy - \gamma \DD \nabla f(\zz))$,  and $\uu = \DD^{-1}(\yy - \prox^{\DD^{-1}}_{\gamma h}(\yy))/ \gamma$ is a minimizer of the dual objective.
\end{corollary}
\begin{proof}
  This follows from the first part of the proof of Theorem~\ref{thm:fixed_point}. In particular, Eq.~\eqref{eq:first_order_optimality_primal} shows that $\zz = \prox^{\DD^{-1}}_{\gamma h}(\yy)$ is a minimizer of the primal objective, while Eq.~\eqref{eq:first_order_optimality_dual} shows that $\uu = \DD^{-1}(\yy - \prox^{\DD^{-1}}_{\gamma h}(\yy))/ \gamma$ is a minimizer of the dual objective.

  The identity $\prox^{\DD^{-1}}_{\gamma h}(\yy) = \prox^{\DD^{-1}}_{\gamma h}(2 \zz - \yy - \gamma \DD \nabla f(\zz))$ comes from the definition of fixed point \eqref{eq:y_ty_0}.
\end{proof}

\clearpage

\section{Iteration complexity analysis}\label{apx:proofs}

In this section we provide a proof for the convergence rate analysis of the proposed methods of \S\ref{scs:analysis}. We will start by with the proof of linear convergence (Theorem~\ref{thm:linear_convergence}) and then prove the sublinear convergence rate (Theorem~\ref{thm:sublinear_convergence}), as this last theorem reuses many elements from the first.

Unless explicitly stated (e.g., in Theorem \ref{thm:sublinear_convergence}), the results are only proven for the sparse variants. Since the dense variants are a special case of the sparse variants with $\PP_i = \boldsymbol{I}$, $\DD = \boldsymbol{I}$, the results for the dense variants follow as a special case.




\vspace{0.5em}

\paragraph{Structure of this appendix.}
\begin{itemize}[leftmargin=*]
\item \ref{scs:analysis_preliminaries} provides technical lemmas that will be used in later proofs.
\item \ref{apx:analysis_linear_convergence} provides a proof for the linear convergence (under assumptions) of the proposed methods (Theorem~\ref{thm:linear_convergence}).
\item \ref{apx:sublinear_dense} provides a sublinear convergence rate for the dense variants of the proposed methods (Theorem~\ref{thm:sublinear_convergence}).
\item \ref{apx:sublinear_sparse} provides a (weaker) sublinear convergence rate for the sparse variants of the proposed methods (Theorem~\ref{thm:sublinear_convergence_sparse}).
\end{itemize}

\paragraph{Extra notation for this section.}
\begin{itemize}[leftmargin=*]
\item We define $f_i(\xx) = \psi_i(\xx) + \omega(\xx)$
\item To provide a unified analysis of the dense and sparse algorithm, we define the following auxiliary function:
\begin{equation}\label{eq:def_xi}
    \xi_i(\xx) \defas \psi_i(\xx) + \sum_{B \in T_i}d_B \omega_B([\xx]_B)~.
\end{equation}
Note that $\frac{1}{n}\sum_{i=1}^n \xi_i = f$. Since $\psi_i$ is $L_\psi$-smooth and $\omega$ is $L_\omega$-smooth we have that $\xi_i$ is $L_f$-smooth, with $L_f = L_\psi + d_{\max}L_\omega$ (as defined in \S\ref{scs:analysis}).
\item Contrary to full gradient algorithms, in stochastic variance reduced methods the objective function is not guaranteed to decrease at each iteration. To compensate for this, a common approach is to add a positive term that decreases throughout the iterations. The resulting function is often called a \emph{Lyapunov} function. Throughout this paper, the positive term that we will add is the following:
\begin{equation}\label{eq:def_H_t}
    H_t \defas \frac{1}{n}\sum_{i=1}^n m_{i, t}
\end{equation}
where $m_{i, t}$ are positive constants initialized as
\begin{equation}
 m_{i, 0} = \frac{1}{2L_f}\|\balpha_{i, 0} - \nabla \psi_i(\xx^\star)\|^2
\end{equation}
and updated at each iteration as
\begin{equation}
    m_{i, t+1} = \begin{cases} B_{f_i}(\zz_t, \xx^\star) &\text{ if $\balpha_i$ has been updated} \\
    m_{i, t} &\text{ otherwise}\end{cases}~,
\end{equation}
 for all $i \in \{1, \ldots, n\}$. This term is a hybrid between those used by \citet{defazio2014saga} and \citet{hofmann2015variance}. Like \citet{defazio2014saga}, it will allow us to obtain a large $1/(3L_f)$ step size, contrary to the $<  1/4 L_f$ step size of \citet{hofmann2015variance}. Like \citet{hofmann2015variance} (and unlike \citet{defazio2014saga}), it will allow to initialize $\balpha_0$ arbitrarily.
 \item For convenience, we denote by $\langle \cdot, \cdot \rangle_{(i)}$ (resp. $\|\cdot\|_{(i)}$) the scalar product (resp. norm) restricted to blocks in the extended support, i.e., $\langle \xx, \yy \rangle_{(i)} \defas \langle \xx, \PP_i \yy \rangle$ and $\|\xx\|_{(i)} \defas \sqrt{\langle \xx, \xx\rangle_{(i)}}$.
 \item We denote by $d_{\max}$ the maximum entry in the diagonal matrix $\DD$, with $\DD$ as defined in \S\ref{scs:sparse}.
  \end{itemize}

\subsection{Preliminaries}\label{scs:analysis_preliminaries}

In this subsection we state some key lemmas that are used in both the proof of linear and sublinear convergence.

\begin{lemma}[Strong convexity inequality]\label{lemma:strongly_convex_bound}
Let $\psi_i$ be $\mu_\psi$-strongly convex. Let $\omega$ be $\mu_\omega$-strongly convex (where we allow $\mu_\psi = \mu_\omega=0$). Then with $\mu = \mu_\psi + \mu_\omega$ we have the following inequality for arbitrary $\xx$ and $\yy$  in the domain:
\begin{equation}
\begin{aligned}
f(\xx) &\geq f(\yy) + \langle \nabla f(\yy), \xx - \yy \rangle + \frac{1}{2n(L_{f}- \mu)}\sum_{i=1}^n\|\nabla \xi_i(\xx) - \nabla \xi_i(\yy)\|^2\\
&\qquad + \frac{\mu L_f}{2 d_{\text{max}}(L_{f}- \mu)}\|\xx - \yy\|^2 + \frac{\mu}{L_{f}- \mu}\langle \nabla f(\xx) - \nabla f(\yy), \yy - \xx \rangle
\end{aligned}
\end{equation}
\end{lemma}
\begin{proof}
We start by proving that $\xi_i$ is $\mu$-strongly convex when restricted to his support. Let $\boldsymbol{a}, \boldsymbol{b}$ be arbitrary vectors in $\RR^{p}$. Then we have the following sequence of inequalities:
\begin{align}
	\langle \nabla &\xi_i(\boldsymbol{a}) - \nabla \xi_i(\boldsymbol{b}), \boldsymbol{a} - \boldsymbol{b}\rangle \\
	&= \langle \nabla \psi_i(\boldsymbol{a}) + \PP_i \DD \nabla \omega(\boldsymbol{a}) - \nabla \psi_i(\boldsymbol{b}) - \PP_i \DD \nabla \omega(\boldsymbol{b}), \boldsymbol{a} - \boldsymbol{b} \rangle\\
	&\qquad \text{ (by definition of $\nabla \xi_i$)}\nonumber\\
	&= \langle \nabla \psi_i(\boldsymbol{a})  - \nabla \psi_i(\boldsymbol{b}), \boldsymbol{a} - \boldsymbol{b} \rangle + \langle \PP_i \DD \nabla \omega(\boldsymbol{a})  - \PP_i \DD \nabla \omega(\boldsymbol{b}), \boldsymbol{a} - \boldsymbol{b} \rangle\\
	&\geq \mu_\psi\|\boldsymbol{a} - \boldsymbol{b}\|^2 + \langle \PP_i \DD \nabla \omega(\boldsymbol{a})  - \PP_i \DD \nabla \omega(\boldsymbol{b}), \boldsymbol{a} - \boldsymbol{b} \rangle\\
	&\qquad \text{ (by strong convexity of $\nabla \psi_i$)}\nonumber\\
	&= \mu_\psi\|\boldsymbol{a} - \boldsymbol{b}\|^2 +  \textstyle\sum_{B \in \mathcal{B}}d_B \langle \nabla \omega_B([\boldsymbol{a}]_B)  - \nabla \omega_B([\boldsymbol{b}]_B), [\boldsymbol{a}]_B - [\boldsymbol{b}]_B\rangle\\
	&\qquad \text{ (by block separability of $\omega$ and definition of $\PP_i \DD$)}\nonumber\\
	&\geq \mu_\psi\|\boldsymbol{a} - \boldsymbol{b}\| + \textstyle\sum_{B \in \mathcal{B}} d_B \mu_\omega\|[\boldsymbol{a}]_B - [\boldsymbol{b}]_B\|^2\\
	&\qquad \text{ (strong convexity of $\omega_B$, consequence of strong cvx of $\omega$)}\nonumber\\
	&\geq \mu_\psi\|\boldsymbol{a} - \boldsymbol{b}\|^2 + \mu_\omega\|\boldsymbol{a} - \boldsymbol{b}\|_{(i)}^2\\
	&\qquad \text{ (using $d_B \geq 1$ by definition)}\nonumber\\
	&\geq \underbrace{\left(\mu_\psi + \mu_\omega\right)}_{=\mu}\|\boldsymbol{a} - \boldsymbol{b}\|_{(i)}^2~.
\end{align}

We have proved that $\xi_i$ is $\mu$-strongly convex  on the subspace generated by the extended support (i.e., with respect to the norm $\|\cdot\|_{(i)}$). Since it is also $L_f$-smooth by \eqref{eq:def_xi}, we can apply~\citep[Lemma 4]{defazio2014saga} to obtain the following inequality, valid for all $\boldsymbol{a}$ and $\boldsymbol{b}$ in its domain:
\begin{equation}
	\begin{aligned}
	\xi_i(\boldsymbol{a}) &\geq \xi_i(\boldsymbol{b}) + \langle\nabla \xi_i(\boldsymbol{b}),  \boldsymbol{a} - \boldsymbol{b} \rangle + \frac{1}{2(L_{f}- \mu)}\|\nabla \xi_i(\boldsymbol{a}) - \nabla \xi_i(\boldsymbol{b})\|^2\\
	&\qquad + \frac{\mu L_f}{2 (L_{f}- \mu)}\|\boldsymbol{a} - \boldsymbol{b}\|_{(i)}^2 + \frac{\mu}{L_{f}- \mu}\langle \nabla \xi_i(\boldsymbol{a}) - \nabla \xi_i(\boldsymbol{b}), \boldsymbol{b} - \boldsymbol{a}\rangle
	\end{aligned}
\end{equation}
We will apply the previous inequality at $\boldsymbol{a} = \xx$, $\boldsymbol{b} = \yy$ and average over all $i$. Note that $\frac{1}{n}\sum_{i=1}^n \xi_i(\xx) = \psi(\xx) + \omega(\xx) = f(\xx)$ by definition of $\xi_i$ and so we can write
\begin{align}
f(\xx) &\geq f(\yy) + \frac{1}{n}\sum_{i=1}^n\langle \nabla \xi_i(\yy),\xx -\yy \rangle + \frac{1}{2n(L_{f}- \mu)}\sum_{i=1}^n\|\nabla \xi_i(\xx) - \nabla \xi_i(\yy)\|^2\nonumber\\
&\qquad + \frac{\mu L_f}{2 (L_{f}- \mu)}\frac{1}{n}\sum_{i=1}^n\|\yy -\xx\|_{(i)}^2 + \frac{\mu}{L_{f}- \mu}\frac{1}{n}\sum_{i=1}^n\langle \nabla \xi_i(\xx) - \nabla \xi_i(\yy), \yy -\xx\rangle	\label{eq:strong_cvx_nu}
\end{align}
We can simplify the terms in this inequality as follows:
\begin{gather}
 \frac{1}{n}\sum_{i=1}^n\langle \nabla \xi_i(\yy), \xx - \yy \rangle = \langle \frac{1}{n}\sum_{i=1}^n \nabla \xi_i(\yy), \xx - \yy \rangle = \langle\nabla f(\yy), \xx - \yy \rangle\\
 \frac{1}{n}\sum_{i=1}^n\|\xx - \yy\|_{(i)}^2 = \|\xx - \yy\|^2_{\DD^{-1}} \geq \frac{1}{d_{\text{max}}}\|\xx - \yy\|^2\\
 \frac{1}{n}\sum_{i=1}^n\langle \nabla \xi_i(\xx) - \nabla \xi_i(\yy), \yy - \xx \rangle	= \langle \nabla f(\xx) - \nabla f(\yy), \yy - \xx\rangle
\end{gather}
The second equality results by the definition of $\DD$ which gives: $\Econd[\boldsymbol{P}_i] = \DD^{-1}$.Using the previous identities (and inequality) into \eqref{eq:strong_cvx_nu} we finally obtain the desired bound:
\begin{equation}
\begin{aligned}
f(\xx) &\geq f(\yy) + \langle \nabla f(\yy), \xx - \yy \rangle + \frac{1}{2n(L_{f}- \mu)}\sum_{i=1}^n\|\nabla \xi_i(\xx) - \nabla \xi_i(\yy)\|^2\\
&\qquad + \frac{\mu L_f}{2 d_{\text{max}}(L_{f}- \mu)}\|\xx - \yy\|^2 + \frac{\mu}{L_{f}- \mu}\langle \nabla f(\xx) - \nabla f(\yy), \yy - \xx \rangle
\end{aligned}
\end{equation}
\end{proof}

\begin{lemma}[Bound on gradient estimate variance] \label{lemma:variance_bound}
Let $\Econd$ denote the conditional expectation with respect to the random index $i$ selected at the $t$-th iteration. Then we have the following inequality:
\begin{equation}
\begin{aligned}
\Econd \|\vv_t - \PP_i\DD\nabla f(\xx^\star)\|^2 &\leq (1 + \beta^{-1}) 2  L_f {H}_t+ (1 + \beta)\Econd\| \nabla\xi_i(\zz_t) - \nabla\xi_i(\xx^\star)\|^2\\
&\qquad- \beta\|\nabla f(\zz_t) - \nabla f(\xx^\star)\|^2  ~,
\end{aligned}
\end{equation}
valid for any $\beta > 0$.
\end{lemma}

\begin{proof}
Let $\psi \defas \frac{1}{n}\sum_{j=1}^n \psi_j$. Then we have the following sequence of inequalities:
\begin{align}
&\Econd \|\vv_t - \PP_i \DD\nabla f(\xx^\star)\|^2 =  \Econd \|\underbrace{\nabla\xi_i( \zz_t) - \balpha_{i, t} + \boldsymbol{D}_i\overline{\balpha}_t - \boldsymbol D_i  \nabla f(\xx^\star)}_{=\boldsymbol{\zeta}_i} \|^2\\
&=  \Econd\| \overbrace{\left[- \balpha_{i, t} + \boldsymbol{D}_i\overline{\balpha}_t + \nabla\xi_i(\xx^\star) - \boldsymbol D_i  \nabla f(\xx^\star)\right] + \big[ \nabla\xi_i( \zz_t) - \nabla\xi_i(\xx^\star)}^{\boldsymbol{\zeta}_i} - \overbrace{(\nabla f(\zz_t) - \nabla f(\xx^\star))}^{\Econd \boldsymbol{\zeta}_i}\big]\|^2\nonumber\\
&\qquad + \|\overbrace{\nabla f(\zz_t) - \nabla f(\xx^\star)}^{\Econd \boldsymbol{\zeta}_i}\|^2\\
&\qquad \text{ (by Lemma~\ref{lemma:generalized_variance_decomposition} with $\QQ_i = \boldsymbol{I}$, $\boldsymbol{A} = \boldsymbol{I}$, and where we have also added and substracted $\nabla\xi_i(\xx^\star)$)}\nonumber\\
&\leq (1 + \beta^{-1})\Econd \|\balpha_{i, t}- \psi_i(\xx^\star) - \boldsymbol{D}_i\overline{\balpha}_t  + \boldsymbol D_i  \nabla \psi(\xx^\star)\|^2 \nonumber \\
&\qquad + (1 + \beta)\Econd\| \nabla\xi_i( \zz_t) - \nabla\xi_i(\xx^\star) - \nabla f(\zz_t) + \nabla f(\xx^\star)\|^2+ \|\nabla f(\zz_t) - \nabla f(\xx^\star)\|^2\\
&\qquad \text{ (by Young's inequality $\|\boldsymbol{a} + \boldsymbol{b}\|^2\leq (1 + \beta^{-1})\|\boldsymbol{a}\|^2 + (1 + \beta)\|\boldsymbol{b}\|^2$ and definition of $\xi_i$)}\nonumber\\
&= (1 + \beta^{-1})\Econd\| \balpha_{i, t} - \psi_i(\xx^\star)\|^2 - (1 + \beta^{-1})\Econd\|\overline{\balpha}_t  - \nabla \psi(\xx^\star)\|_{\DD}^2\nonumber\\
&\qquad + (1 + \beta)\Econd\| \nabla\xi_i( \zz_t) - \nabla\xi_i(\xx^\star)\|^2 - \beta\|\nabla f(\zz_t) - \nabla f(\xx^\star)\|^2~,
\end{align}
where in the last equivalence we have applied Lemma~\ref{lemma:generalized_variance_decomposition} both to the first term (with $\QQ_i = \PP_i $, $\boldsymbol{A} = \DD$) and to the second term (this time with $\boldsymbol{A} = \boldsymbol{I}$, $\QQ_i = \boldsymbol{I}$).
In all, and dropping the negative second term we have the inequality
\begin{equation}
\begin{aligned}
	\Econd \|\vv_t - \PP_i \DD\nabla f(\xx^\star)\|^2 &\leq (1 + \beta^{-1})\Econd\| \balpha_{i, t} - \psi_i(\xx^\star)\|^2 + (1 + \beta)\Econd\| \xi_i(\zz_t) - \nabla\xi_i(\xx^\star)\|^2\\
	&\qquad - \beta\|\nabla f(\zz_t) - \nabla f(\xx^\star)\|^2
\end{aligned}\label{eq:variance_bound_proof_1}
\end{equation}

We will now bound the first term of the above inequality. Let $\mathcal{J}$ denote the set of indices for which the memory terms have been updated at least once and $\mathcal{J}^c$ its complement.
For $j \in \mathcal{J}$, we denote by $\bphi_{j , t}$ the iterate at which $\balpha_{j}$ was last updated, i.e., $\balpha_{j, t} = \nabla \psi(\bphi_{j , t})$ for all $j$ and $t$.
Then we have
\begin{align}
	\Econd\|\balpha_{i, t}- \nabla\psi_i(\xx^\star)\|^2  &= \frac{1}{n}\left(\sum_{j \in \mathcal{J}} \|\balpha_{i, t}- \nabla\psi_i(\xx^\star)\|^2 + 2 L_f \sum_{j \in \mathcal{J}^c} \xi_{t, j}\right)\\
	&\leq \frac{2L_f}{n}\left(\sum_{j \in \mathcal{J}} B_{\psi_i}(\bphi_{i, k}, \xx^\star) + \sum_{j \in \mathcal{J}^c} \xi_{t, j}\right)\\
	& \qquad \text{(by Lemma \ref{lemma:l_smooth_ineq} and also using $L_f \geq L_\psi$)}\nonumber \\
	&\leq \frac{2L_f}{n}\left(\sum_{j \in \mathcal{J}} B_{f_i}(\bphi_{i, k}, \xx^\star) + \sum_{j \in \mathcal{J}^c} \xi_{t, j}\right)\\
&\qquad \text{ (adding $B_\omega$, which is positive by convexity of $\omega$)}\nonumber\\
    &= 2 L_f H_t~.
\end{align}

Finally, plugging this bound back in \eqref{eq:variance_bound_proof_1} we obtain the desired inequality:
\begin{equation}
\begin{aligned}
\Econd \|\vv_t - \PP_i \DD\nabla f(\xx^\star)\|^2 &\leq (1 + \beta^{-1}) 2L_{f}H_t + (1 + \beta)\Econd\| \nabla\xi_i(\zz_t) - \nabla\xi_i(\xx^\star)\|^2\\
&\qquad - \beta\|\nabla f(\zz_t) - \nabla f(\xx^\star)\|^2
\end{aligned}
\end{equation}
\end{proof}

\begin{lemma}[Evolution of $H_t$]\label{lemma:evolution_hk}
Let $\Econd$ denote the conditional expectation with respect to the random index $i$ selected at the $t$-th iteration. Then for every iteration $t \geq 0$ we have (with $q=1$ for \SAGA\ variants):
\begin{equation}
	\Econd H_{t+1} = \frac{q}{n} B_f(\zz_t, \xx^\star) + \left(1 - \frac{q}{n}\right)H_t~.
\end{equation}

\begin{proof}
By definition of $m_{j, t+1}$ in Eq.~\eqref{eq:def_xi}, for a fixed index $j$ we have:
\begin{equation}
\Econd[m_{j, t+1}] = \frac{q}{n}B_{f_j}(\zz_t, \xx^*) + (1 - \frac{q}{n}) m_{j, t}~.
\end{equation}
Hence averaging over all indices we get
\begin{align}
\Econd[H_{t+1}] &= \frac{1}{n}\sum_{j=1}^n \Econd[m_{j, t+1}]& \nonumber \\
&= \frac{1}{n}\sum_{j=1}^n \left(\frac{q}{n} B_{f_j}(\zz_t, \xx^*) + (1 - \frac{q}{n})m_{j, t}\right)&\nonumber \\
&= \frac{q}{n} B_f(\zz_t,\xx^*) + (1 - \frac{q}{n})H_t&
\end{align}
\end{proof}

\end{lemma}

\clearpage

\subsection{Linear convergence: proof of Theorem~\ref{thm:linear_convergence}}\label{apx:analysis_linear_convergence}

The {\bfseries proof is structured} as follows:
\begin{itemize}[leftmargin=*]
\item We start by proving an inequality that relates $\|\yy_{t+1} - \yy^\star\|^2$ with $\|\yy_t - \yy^\star\|^2$, where $\yy^\star$ is a fixed point of $\boldsymbol{G}_\gamma$. This inequality will be central in both proofs of linear and sublinear convergence. We call this the ``master recurrence inequality'' (Lemma~\ref{lemma:master_inequality}),
\item As is often the case in variance reduced methods, a recurrence purely in terms of the iterates as the one in Lemma~\ref{lemma:master_inequality} does not provide the monotonic decrease required to prove a linear convergence rate. To overcome this, we will make use of an auxiliary function which is always larger than the suboptimality criterion and which \emph{does} verify a monotonic decrease in expectation. This is often referred to as a \emph{Lyapunov} function.
The Lyapunov function that we will use is the following:
\begin{equation}\label{eq:def_lyapunov}
V_t \defas c \|\yy_t - \yy^\star\|^2 + {H}_t \quad,
\end{equation}
with $H_t$ as defined in \eqref{eq:def_H_t} and $\yy^\star$ an arbitrary fixed point of $\boldsymbol{G}_\gamma$.

\item Finally, in Theorem \ref{thm:linear_convergence} we use the decrease of the Lyapunov function prove the desired rates of convergence.
\end{itemize}

\hfill


\begin{lemma}[Master recurrence inequality]\label{lemma:master_inequality}
Let $\{\yy_t, \xx_t, \zz_t\}$ be the iterates produced by any of the proposed algorithms, $\yy^\star \in \Fix(\boldsymbol{G}_\gamma)$
and $\xx^\star \defas \prox^{\DD^{-1}}_{\gamma h}(\yy^\star)$ (with $\DD = \boldsymbol{I}$ for the dense variants).
Then we have the following inequality, valid for all $\beta > 0$ and $s>0$:
\begin{align}
    \Econd\|\yy_{t+1} - \yy^\star\|^2 &\leq \|\yy_t - \yy^\star\|^2 + (s - 1)\Econd \|\yy_{t+1} - \yy_t\|^2 \nonumber\\
&\qquad+ \frac{\gamma^2}{s}(1 + \beta^{-1}) 2L_f H_t \nonumber\\
&\qquad + (\frac{\gamma^2}{s}(1 + \beta) - \frac{\gamma}{L_f})\Econd\| \nabla\xi_i(\zz_t) - \nabla \xi_i(\xx^\star)\|^2 \nonumber\\
&\qquad   +(- 2\frac{\gamma^2 \beta}{s}\mu - 2\gamma \frac{L_f-\mu}{L_f}) B_f(\zz_t,\xx^\star)  \nonumber\\
&\qquad - \frac{\gamma \mu}{d_{\max}}\|\zz_t - \xx^\star\|^2
\end{align}
\end{lemma}
\begin{proof}
Developing the square we have
\begin{align}
\Econd\|\yy_{t+1} - \yy^\star\|^2 &= \Econd\|\yy_t + (\yy_{t+1} - \yy_t) - \yy^\star\|^2\\
&= \|\yy_t - \yy^\star\|^2 + \Econd\|\yy_{t+1} - \yy_t\|^2 + 2\Econd \langle \yy_{t+1} - \yy_t, \yy_t - \yy^\star\rangle~.\label{eq:initial_decomp_master}
\end{align}
We will now work towards bounding the last term of this expression.

Let $i$ denote the random index selected at iteration $t$.
Note that by definition of $\DD$ we have $\Econd[\boldsymbol{P}_i] = \DD^{-1}$ and so we can write:
\begin{align}
  \Econd \langle \zz_t - \xx^\star, \boldsymbol{P}_i(\yy_t - {\yy^\star})\rangle
  &= \langle \zz_t - \xx^\star, \yy_t - {\yy^\star}\rangle_{\DD^{-1}}\\
  & \stackrel{}{\geq} \|\zz_t - \xx^\star\|_{\DD^{-1}}^2~,
\end{align}
where in the last inequality we have used Eq.~\eqref{eq:nonexpansive} with $\zz_t = \prox^{\DD^{-1}}_{\gamma h}(\yy_t)$ and $\xx^\star = \prox_{\gamma g}^{\DD}(\yy^\star)$. Using once again  the identity $\Econd[\boldsymbol{P}_i] = \DD^{-1}$ and noting that $\zz_t$ does not depend on $i$  we have $ \|\zz_t - \xx^\star\|_{\DD^{-1}}^2 = \Econd \|\zz_t - \xx^\star\|^2_{\PP_i}$ and so in all, we have
\begin{equation}\label{eq:strict_nonexpansiveness_h}
\Econd \langle \zz_t - \xx^\star, \PP_i(\yy_t - \yy^\star)\rangle \geq \Econd \|\zz_t - \xx^\star\|_{\PP_i}^2~.
\end{equation}
Furthermore, by the blockwise version of the firm non-expansiveness of the prox (Lemma~\ref{lemmma:block_nonexpansive}), from the definition of $\xx_t$ in \VRTOS{} we also have the following inequality, with $\xx_t = \prox^{\DD^{-1}}_{\gamma g}( 2\zz_t - \yy_t - \gamma \vv_t)$ and $\xx^\star = \prox^{\DD^{-1}}(2 \xx^\star - \yy^\star - \gamma \DD \nabla f({\xx^\star}))$, where this last equality is a consequence of Colollary~\ref{cor:fixed_point}:
\begin{equation}
     \langle 2 \zz_t - \yy_t - \gamma \vv_t - 2 \xx^\star + \yy^\star + \gamma \PP_i \DD \nabla f({\xx^\star}) , \PP_i(\xx_t - {\xx^\star})\rangle - \|\xx_t - {\xx^\star}\|_{\PP_i}^2  \geq 0~,
\end{equation}
which taking conditional expectation gives
\begin{equation}\label{eq:strict_nonexpansiveness_g}
     \Econd\langle 2 \zz_t - \yy_t - \gamma \vv_t - 2 \xx^\star + \yy^\star + \gamma \PP_i \DD \nabla f({\xx^\star}) , \PP_i(\xx_t - {\xx^\star})\rangle - \Econd\|\xx_t - {\xx^\star}\|_{\PP_i}^2  \geq 0~.
\end{equation}

We now have the following sequence of inequalities:
\begin{align}
&\Econd\langle \yy_{t+1} - \yy_t, \yy_t - \yy^\star\rangle \\
&\quad= \Econd\langle \xx_t - \zz_t, \PP_i(\yy_t - \yy^\star)\rangle\\
&\quad\qquad \text{ (definition of $\yy_{t+1}$)}\nonumber\\
&\quad=  \Econd \langle \xx_t - \zz_t - \xx^\star + \xx^\star, \PP_i(\yy_t - \yy^\star)\rangle\\
&\quad \qquad \text{ (adding and substracting $\xx^\star$)}\nonumber\\
&\quad=  \langle \xx_t - \xx^\star, \PP_i(\yy_t - \yy^\star) \rangle - \Econd\langle \zz_t - \xx^\star, \PP_i(\yy_t - \yy^\star)\rangle \\
&\quad\leq  \Econd \langle \xx_t - \xx^\star, \PP_i(\yy_t - \yy^\star)\rangle - \Econd \|\zz_t - \xx^\star\|_{\PP_i}^2 \qquad \text{ (by Eq~\eqref{eq:strict_nonexpansiveness_h})} \\
&\quad\leq \Econd\langle 2 \zz_t - \gamma \vv_t - 2 \xx^\star + \gamma \DD \nabla f(\xx^\star), \PP_i(\xx_t - \xx^\star) \rangle\!-\!\Econd\|\xx_t - \xx^\star\|_{\PP_i}^2\!-\!\Econd\|\zz_t - \xx^\star\|_{\PP_i}^2\\
&\quad \qquad \text{ (adding Eq.~\eqref{eq:strict_nonexpansiveness_g})}\nonumber\\
&\quad\leq \Econd\Big[\langle 2 \zz_t - \gamma \vv_t - 2 \xx^\star + \gamma \DD \nabla f(\xx^\star), \PP_i(\xx_t - \xx^\star) \rangle - \|\xx_t - \xx^\star\|_{\PP_i}^2\\
&\qquad\qquad\qquad -  \|\zz_t - \xx^\star\|_{\PP_i}^2\Big]\quad \qquad \text{ (by linearity of expectation)}\nonumber\\
&\quad\leq \Econd\Big[- \left(\|\xx_t - \xx^\star\|_{\PP_i}^2 - 2\langle \zz_t - \xx^\star, \PP_i(\xx_t - \xx^\star) \rangle  +  \|\zz_t - \xx^\star\|_{\PP_i}^2\right)\\
 &\qquad\qquad\qquad + \langle - \gamma \vv_t + \gamma \DD \nabla f(\xx^\star), \PP_i(\xx_t - \xx^\star) \rangle\Big]\quad\qquad \text{ (reordering terms)}\nonumber\\
&\quad= \Econd\left[- \|\zz_t - \xx^\star - (\xx_t - \xx^\star)\|_{\PP_i}^2 - \gamma  \langle  \vv_t -  \PP_i\DD \nabla f(\xx^\star), \xx_t - \xx^\star \rangle\right]\\
&\quad \qquad \text{ (completing the square)}\nonumber\\
&\quad= \Econd\Big[- \|\zz_t - \xx_t\|_{\PP_i}^2 -  \langle  \gamma \vv_t - \gamma {\PP_i}\DD \nabla f(\xx^\star), \xx_t - \zz_t \rangle\nonumber\\
&\qquad\qquad- \langle \gamma \vv_t - \gamma \PP_i \DD \nabla f(\xx^\star), \zz_t - \xx^\star \rangle \Big] \qquad\text{ (adding and substracting $\zz_t$)}\nonumber\\
&\quad\leq \Big(\frac{s}{2} - 1\Big)\Econd \|\yy_{t+1} - \yy_t\|^2 + \frac{\gamma^2}{2s}\Econd\| \vv_t -  \PP_i\DD \nabla f(\xx^\star) \|^2\nonumber\\
&\qquad\qquad- \gamma \langle   \nabla f(\zz_t) -  \nabla f(\xx^\star), \zz_t - \xx^\star \rangle~,\label{eq:last_ineq_y_y_y_y}
    \end{align}
where in the last inequality we have used Young's inequality: $|\langle \aa, \boldsymbol{b}] \rangle | \leq \frac{s}{2}\|\aa\|^2 + \frac{1}{2 s}\|\boldsymbol{b}\|^2$ and definition of $\yy_{t+1}$ for the first term and computed the expectation in the last term.

Replacing this last inequality into \eqref{eq:initial_decomp_master} we obtain
\begin{equation}
\begin{aligned}
    \Econd\|\yy_{t+1} - \yy^\star\|^2 &\leq \|\yy_t - \yy^\star\|^2 + (s-1) \Econd \|\yy_{t+1} - \yy_t\|^2\\
    & +\frac{\gamma^2}{s}\Econd\| \vv_t -  {\PP_i}\DD \nabla f(\xx^\star) \|^2 - 2\gamma \langle   \nabla f(\zz_t) -  \nabla f(\xx^\star), \zz_t - \xx^\star \rangle
    \end{aligned}
\end{equation}

We will proceed to further bound the second and last terms using previous results.
For the second term, we can use the bound $\Econd \|\vv_t - \PP_i \PP_i \DD \nabla f(\xx^\star)\|^2  \leq (1 + \beta^{-1}) 2L_f H_t + {(1 + \beta)\Econd\| \nabla\xi_i(\zz_t) - \nabla \xi_i(\xx^\star)\|^2}  - \beta\|\nabla f(\zz_t) - \nabla f(\xx^\star)\|^2$ from Lemma~\ref{lemma:variance_bound}, giving:
\begin{equation}
\begin{aligned}
\frac{\gamma^2}{2s}\Econd\| \vv_t -  \DD \nabla f(\xx^\star) \|_{(i)}^2 &\leq \frac{\gamma^2}{s}(1 + \beta^{-1}) L_f H_t + \frac{\gamma^2}{2s}(1 + \beta)\Econd\| \nabla\xi_i(\zz_t) - \nabla \xi_i(\xx^\star)\|^2\\
&\qquad   - \frac{\gamma^2 \beta}{2s}\|\nabla f(\zz_t) - \nabla f(\xx^\star)\|^2	~,
\end{aligned}
\end{equation}
The third term can be bounded using the strong convexity inequality  of Lemma~\ref{lemma:strongly_convex_bound} with $\yy = \zz_t, \xx = \xx^\star$, to obtain
\begin{align}
  -&\gamma \langle \nabla f(\zz_t) - \nabla f(\xx^\star), \zz_t - \xx^\star\rangle =  \gamma\langle\nabla f(\xx^\star), \zz_t - \xx^\star\rangle + \gamma \langle \nabla f(\zz_t), \xx^\star - \zz_t\rangle\\
  &\leq \gamma \langle\nabla f(\xx^\star), \zz_t - \xx^\star\rangle + \gamma (\frac{L_{f}-\mu}{L_{f}}(f(\xx^\star) - f(\zz_t))\nonumber \\
  &\qquad - \frac{1}{2L_{f}}\Econd\|\nabla\xi_i( \zz_t) - \nabla\xi_i(\xx^\star)\|^2 - \frac{\mu}{2d_{\max}}\|\zz_t - \xx^\star\|^2 - \frac{\mu}{L_{f}}\langle \nabla f(\xx^\star), \zz_t - \xx^\star\rangle)\\
  &\leq - \gamma \frac{L_{f}-\mu}{L_{f}}B_f(\zz_t, \xx^\star) + \gamma\left(- \frac{1}{2L_{f}}\Econd\|\nabla\xi_i( \zz_t) - \nabla\xi_i(\xx^\star)\|^2 - \frac{\mu}{2d_{\max}}\|\zz_t - \xx^\star\|^2\right)
\end{align}
Using the bound for these two terms in \eqref{eq:last_ineq_y_y_y_y} we have
\begin{align}
    \Econd\|\yy_{t+1} - \yy^\star\|^2 &\leq  \|\yy_t - \yy^\star\|^2 + (s - 1)\Econd \|\yy_{t+1} - \yy_t\|^2  \nonumber\\
&\qquad \text{ ($\yy_t$ and $\yy^\star$ do not depend on $i$)}\nonumber \\
&\qquad+ \frac{\gamma^2}{s}(1 + \beta^{-1}) 2L_f H_t  \nonumber\\
&\qquad + (\frac{\gamma^2}{s}(1 + \beta) - \frac{\gamma}{L})\Econd\| \nabla\xi_i(\zz_t) - \nabla \xi_i(\xx^\star)\|^2  \nonumber\\
&\qquad   - \frac{\gamma^2 \beta}{s}\|\nabla f(\zz_t) - \nabla f(\xx^\star)\|^2 \nonumber\\
&\qquad - 2\gamma \frac{L-\mu}{L}B_f(\zz_t, \xx^\star) - \frac{\gamma \mu}{d_{\max}}\|\zz_t - \xx^\star\|^2
\end{align}
We now use the bound $-\|\nabla f(\zz_t) - \nabla f(\xx^\star)\|^2 \leq -2\mu B_f(\zz_t,\xx^\star)$ \citep[Theorem 2.1.10]{nesterov2004introductory} to obtain:
\begin{align}
    \Econd\|\yy_{t+1} - \yy^\star\|^2 &\leq \|\yy_t - \yy^\star\|^2 + (s - 1)\Econd \|\yy_{t+1} - \yy_t\|^2  \nonumber\\
&\qquad+ \frac{\gamma^2}{s}(1 + \beta^{-1}) 2L_fH_t  \nonumber\\
&\qquad + (\frac{\gamma^2}{s}(1 + \beta) - \frac{\gamma}{L})\Econd\| \nabla\xi_i(\zz_t) - \nabla \xi_i(\xx^\star)\|^2  \nonumber\\
&\qquad   +(- 2\frac{\gamma^2 \beta}{s}\mu - 2\gamma \frac{L-\mu}{L}) B_f(\zz_t,\xx^\star)  \nonumber\\
&\qquad - \frac{\gamma \mu}{d_{\max}}\|\zz_t - \xx^\star\|^2
\end{align}
\end{proof}

\begin{lemma}[Lyapunov inequality]\label{lemma:lyapunov}
Let $\{\yy_t, \xx_t, \zz_t\}$ be the iterates produced by any of the proposed algorithms for $t \geq 0$, $ \yy^\star \in \Fix(\boldsymbol{G}_\gamma)$
and $\xx^\star \defas \prox^{\DD^{-1}}_{\gamma h}(\yy^\star)$. Let the Lyapunov function $V_t$ be as defined in \eqref{eq:def_lyapunov}.
  Then we have the following inequality:
\begin{equation}
    \begin{aligned}
        \Econd V_{t+1} &\leq V_t + c(s - 1)\Econd \|\yy_{t+1} - \yy_t\|^2 +  \left(\frac{2L_f c\gamma^2}{s}(1 + \beta^{-1}) - \frac{q}{n}\right) H_t \\
        &\qquad + c\left(\frac{\gamma^2}{s}(1 + \beta) - \frac{\gamma}{L_f}\right)\Econd\| \nabla\xi_i(\zz_t) - \nabla \xi_i(\xx^\star)\|^2\\
&\qquad + \left(- 2\frac{c\gamma^2 \beta}{s}\mu - 2c\gamma \frac{L_f-\mu}{L_f} + \frac{q}{n}\right) B_f(\zz_t,\xx^\star) - c\frac{\gamma \mu}{d_{\max}}\|\zz_t - \xx^\star\|^2\\\\
    \end{aligned}
\end{equation}
\end{lemma}
\begin{proof}
  We will first compute the conditional expectation of the Lyapunov term,
  $
\Econd V_{t+1}
  $.
For the first term we can use the bound in Lemma~\ref{lemma:master_inequality} to obtain
\begin{align}
c\Econd\|\yy_{t+1} - \yy^\star\|^2 &\leq c\|\yy_t - \yy^\star\|^2 + c(s - 1)\Econd \|\yy_{t+1} - \yy_t\|^2  \nonumber\\
&\qquad+ \frac{\gamma^2}{s}(1 + \beta^{-1}) 2c L_f H_t  \nonumber\\
&\qquad + c(\frac{\gamma^2}{s}(1 + \beta) - \frac{\gamma}{L_f})\Econd\| \nabla\xi_i(\zz_t) - \nabla \xi_i(\xx^\star)\|^2  \nonumber\\
&\qquad   +c(- 2\frac{\gamma^2 \beta}{s}\mu - 2\gamma \frac{L_f-\mu}{L_f}) B_f(\zz_t,\xx^\star)  \nonumber\\
&\qquad - c\frac{\gamma \mu}{d_{\max}}\|\zz_t - \xx^\star\|^2
\end{align}
  Using Lemma~\ref{lemma:evolution_hk}, the second term of the Lyapunov function gives :
  \begin{align}
  \Econd\, {H}_{t+1} = (1 - \frac{q}{n})H_t + \frac{q}{n}B_f(\zz_t, \xx^\star) ~.
    \end{align}
  and so adding both inequalities we have
\begin{align}
    \Econd V_{t+1} &\leq  \overbrace{c\|\yy_t - \yy^\star\|^2 + H_t}^{V_t} + c(s - 1)\Econd \|\yy_{t+1} - \yy_t\|^2  \nonumber\\
&\qquad + (\frac{2L_{f}c\gamma^2}{s}(1 + \beta^{-1}) - \frac{q}{n}) H_t  \nonumber\\
&\qquad + c(\frac{\gamma^2}{s}(1 + \beta) - \frac{\gamma}{L_f})\Econd\| \nabla\xi_i(\zz_t) - \nabla \xi_i(\xx^\star)\|^2  \nonumber\\
&\qquad + (- 2c\frac{\gamma^2 \beta}{s}\mu - 2c\gamma \frac{L_f - \mu}{L_f} + \frac{q}{n}) B_f(\zz_t,\xx^\star)  \nonumber\\
&\qquad - c\frac{\gamma \mu}{d_{\max}}\|\zz_t - \xx^\star\|^2
\end{align}
  which completes the proof.
\end{proof}

\begin{mdframed}
\begin{customtheorem}{\ref{thm:linear_convergence}}
\TheoremLinear
\end{customtheorem}
\end{mdframed}

\begin{proof}
  %
  From the Lyapunov inequality of Lemma~\ref{lemma:lyapunov} with $s = 1$ we have the following sequence of inequalities
  \begin{align}
      &\Econd V_{t+1} - (1 - \rho)V_t \leq \nonumber\\
      &\qquad\rho V_t + \left({2Lc\gamma^2}(1 + \beta^{-1}) - \frac{q}{n}\right) H_t + c\left({\gamma^2}(1 + \beta) - \frac{\gamma}{L_f}\right)\Econd\| \nabla\xi_i(\zz_t) - \nabla \xi_i(\xx^\star)\|^2\nonumber\\
&\quad  + \left(- 2{c\gamma^2 \beta}\mu - 2c\gamma \frac{L_f-\mu}{L_f} + \frac{q}{n}\right) B_f(\zz_t,\xx^\star) - c\frac{\gamma \mu}{d_{\max}}\|\zz_t - \xx^\star\|^2  \\
&\leq \rho \left(c\|\yy_t - \yy^\star\|^2 + H_t\right)  + \left(\frac{2L_f c\gamma^2}{s}(1 + \beta^{-1}) - \frac{q}{n}\right) H_t\nonumber\\
&\qquad  + c(\frac{\gamma^2}{s}(1 + \beta) - \frac{\gamma}{L_f})\Econd\| \nabla\xi_i(\zz_t) - \nabla \xi_i(\xx^\star)\|^2 + (- 2\frac{c\gamma^2 \beta}{s}\mu - 2c\gamma \frac{L_f-\mu}{L_f} + \frac{q}{n}) B_f(\zz_t,\xx^\star) \nonumber \\
&\qquad \qquad - c\frac{\gamma \mu }{d_{\max}^{3}(1 + \gamma L_h)^2} \|\yy_{t} - \yy^\star\|^2\label{eq:using_smoothness}\\
&\qquad \qquad \qquad \text{(using Lemma \ref{lemma:properties_prox} on the last term, where we have bounded $d_{\min} \geq 1$)}\nonumber\\
&\leq c\left[\rho - \frac{\gamma \mu }{d_{\max}^{3}(1+ \gamma L_h)^{2}}\right] \|\yy_t - \yy^\star\|^2 + \left[\rho + 2L_{f}c\gamma^2(1 + \beta^{-1}) - \frac{q}{n}\right] H_t  \nonumber\\
&\qquad  + c\gamma \left[\gamma(1 + \beta) - \frac{1}{L_f}\right]\Econd\| \nabla\xi_i(\zz_t) - \nabla \xi_i(\xx^\star)\|^2 \nonumber\\
&\qquad \qquad + \left[- 2c\gamma^2 \beta\mu - 2c\gamma \frac{L_f-\mu}{L_f} + \frac{q}{n}\right] B_f(\zz_t,\xx^\star)
  \end{align}
It is worth noting that Eq.~\eqref{eq:using_smoothness} is the only part of the proof in which we use the smoothness of $h$.

Taking the coefficients
\begin{equation}
c = \frac{q}{2\gamma(1 - \gamma\mu)n}~,\quad \beta = 2~, \quad\rho = \min\left \{ \frac{q}{4n}, \frac{1}{3 d_{\max}\delta^2 \kappa} \right \}~,
\end{equation}
With $\delta = d_{\max}(1 + \frac{L_h}{3L_f})$ and $\kappa = {L_f}/{\mu}$. One can verify that all square brackets are non-positive for $\gamma \leq 1/ (3L_f)$ (the coefficients are the same, except for the first square bracket, than those that appear in \citep[Theorem 1]{defazio2014saga}. We hence have
\begin{equation}
    \Econd V_{t+1} \leq (1 - \rho)V_t~,
\end{equation}
which chaining expectations gives
\begin{equation}\label{eq:linear_conv_y}
    \EE \|\yy_{t+1} - \yy^\star\|^2 \leq \EE V_{t+1} \leq (1 - \rho)^{t+1} V_0~.
\end{equation}
This gives a geometric convergence on $\yy_t$. However, we would like to have a convergence rate in terms of the primal iterate $\xx_t$.

By Theorem \ref{thm:fixed_point} we have that $\xx^\star = \prox^{\DD^{-1}}_{\gamma h}(\yy^\star)$, and in this case the minimizer is unique because of strong convexity. Then by firm nonexpansiveness of the prox (Lemma~\ref{lemma:properties_prox}) we have
  \begin{equation}
       \|\zz_{t+1} - \xx^\star\|_{\DD}^2 \leq \|\yy_{t+1} - {\yy^\star}\|^2_{\DD}
  \end{equation}
  which combined with Lemma~\ref{lemma:bound_matrix_norm} and bounding $d_{\min}$ by $1$ (by definition all diagonal entries in $\DD$ are $\geq 1$) gives
\begin{align}
    \|\zz_{t+1} - \xx^\star\|^2 &\leq \frac{1}{d_{\min}}\|\zz_{t+1} - \xx^\star\|_{\DD}^2 \leq \frac{1}{d_{\min}}
    \|\yy_{t+1} - {\yy^\star}\|^2_{\DD} \\
    &\leq \left(\frac{d_{\max}}{d_{\min}}\right)
    \|\yy_{t+1} - {\yy^\star}\|^2\leq {d_{\max}}
    \|\yy_{t+1} - {\yy^\star}\|^2
\end{align}
Combining this with \eqref{eq:linear_conv_y} gives the following bound in $\zz_{t+1}$:
\begin{equation}
    \EE\|\zz_{t+1} - \xx^\star\|^2\leq {d_{\max}} \Econd
    \|\yy_{t+1} - {\yy^\star}\|^2 \leq (1 - \rho)^{t+1} {d_{\max}} V_0~,
\end{equation}
and the claimed bound follows from definition of $V_0$.

\end{proof}

\clearpage
\subsection{Proof of sublinear convergence rate -- dense algorithms}\label{apx:sublinear_dense}

In this section we give a proof of convergence for the dense variants of the proposed algorithms (Algorithm~\ref{alg:vrtos}).
Because we will not be considering the sparse variants, we assume $\DD = \boldsymbol{I}$ without explicit mention.

\begin{lemma}[Bound on gradient estimate variance, Variant 2]\label{lemma:bound_variance_2}
\begin{equation}
\Econd\|\vv_t - \nabla f(\zz_t)\|^2 \leq (1 +  \eta) \Econd\|\nabla f_i(\zz_t) - \nabla f_i(\xx^\star)\|^2 + 2 (1+ \eta^{-1}) L_{f}H_t
\end{equation}

\begin{proof}
\begin{align}
&\Econd\|\vv_t - \nabla f(\zz_t)\|^2 &\\
& = \Econd\|\nabla \psi_i(\zz_t) - \balpha_{i, t}+ (\overline{\balpha}_t + \nabla \omega(\zz_t)) - \nabla f(\zz_t)\|^2 &\\
& = \Econd\|\underbrace{\nabla f_i(\zz_t) - \balpha_{i, t}}_{\xi} +  \underbrace{\overline{\balpha}_t   - \nabla f(\zz_t)}_{- \Econd\xi}\|^2 &\\
&\qquad \text{ (By definition of $f_i$) }\nonumber&\\
&\leq \Econd\|\nabla f_i(\bphi_i^t) - \nabla f_i(\zz_t)\|^2&\\
&\qquad \text{ (Applying Lemma \ref{lemma:generalized_variance_decomposition} and by definition of $f_i$)}&\nonumber\\
&\leq \Econd\|\nabla f_i(\bphi_i^t) - \nabla f_i(\xx^\star) + \nabla f_i(\xx^\star) - \nabla f_i(\zz_t)\|^2&\\
&\qquad \text{ (Adding and substracting $\nabla f_i(\xx^\star)$)}&\nonumber\\
&\leq (1+ \eta^{-1})\Econd\|\nabla f_i(\bphi_i^t) - \nabla f_i(\xx^\star)\|^2 + (1+ \eta) \Econd\|\nabla f_i(\zz_t) - \nabla f_i(\xx^\star)\|^2&\\
&\qquad \text{ (Applying Young's inequality)}\nonumber\\
&\leq 2L_f (1+ \eta^{-1})H_t + (1+ \eta) \Econd\|\nabla f_i(\zz_t) - \nabla f_i(\xx^\star)\|^2&\\
&\qquad \text{ (Applying Lemma 6 from \citep{defazio2014saga} on the first term)}&\nonumber
\end{align}
\end{proof}
\end{lemma}

\begin{lemma}[Saddle point recursive inequality]\label{theorem:saddle_point_inequality_1}
Let $\gamma \leq 1 / L$ and $\yy_t, \xx_t, \uu_t$ be the iterates generated by either \VRTOS\ (Algorithm~\ref{alg:vrtos}). Then we have the following inequality for any $(\xx, \uu) \in \dom(g)\times \dom(h)$, with $\yy = \xx + \gamma \uu$:
\begin{equation}
\begin{aligned}
&2 \gamma \Econd(\mathcal{L}(\xx_t, \uu) - \mathcal{L}(\xx, \uu_t)) + \Econd\|\yy_{t + 1} - \yy\|^2 \\
&\qquad \leq \Econd\|\yy_{t} - \yy\|^2+ 2 \gamma^2(1 +  \eta) \Econd\|\nabla\xi_i (\zz_t) - \nabla\xi_i(\xx^\star)\|^2 + 4 \gamma^2 (1+ \eta^{-1})L_{f}H_t
\end{aligned}
\end{equation}
\end{lemma}

\begin{proof}
By the convexity and the $L$-smoothness inequality, $f$ verifies the following inequalities for an arbitrary $\xx$:
\begin{align}
  &f(\zz_t) - f(\xx)\leq \langle \nabla f(\zz_t), \zz_t - \xx\rangle \label{eq:f_convex}\\
  &f(\xx_t) - f(\zz_t)\leq \langle \nabla f(\zz_t), \xx_t - \zz_t \rangle + \frac{L}{2}\|\zz_t - \xx_t\|^2~\\
  &f(\xx_t) - f(\xx)\leq \langle \nabla f(\zz_t), \xx_t - \xx \rangle + \frac{L}{2}\|\zz_t - \xx_t\|^2~,
\label{eq:convex_and_l_smooth}
  \end{align}
  where the last equation is derived from adding the previous two.
We now derive inequalities for $g$ and $h^*$. From the subdifferential characterization of the proximal operator (Lemma~\ref{lemma:prox_characterization}), the update $\zz_t = \prox_{\gamma h}(\yy_t) $  implies the inclusion
\begin{equation}
    \uu_{t} = \frac{1}{\gamma}(\yy_t - \zz_t) \in \partial h(\zz_{t}) ~\implies \zz_{t} \in \partial h^*(\uu_{t})
\end{equation}
where the implication is a consequence of the Fenchel-Young inequality, see e.g. \citep[Proposition 16.10]{bauschke2017convex} or \citep[Proposition 11.3]{rockafellar1998variational}. Similarly, the update $\xx_{t} = \prox_{\gamma g}(2 \zz_t - \yy_t - \gamma \vv_t)$ in its turn gives the inclusion
  \begin{align}
  &\frac{1}{\gamma}(2 \zz_t - \yy_t - \gamma\vv_t - \xx_{t}) \in \partial g(\xx_{t})
\end{align}
By convexity of $g$ and $h^*$ we then have the inequalities
\begin{align}
    h^*(\uu_{t}) - h^*(\uu) &\leq \langle \zz_{t}, \uu_{t} - \uu \rangle ~.\label{eq:ineq_g}\\
    g(\xx_t) - g(\xx) &\leq \frac{1}{\gamma} \langle \zz_t  - \xx_t, \xx_t - \xx\rangle - \langle \uu_t +\vv_t, \xx_t - \xx \rangle\label{eq:ineq_h}
\end{align}
Adding \eqref{eq:convex_and_l_smooth} and \eqref{eq:ineq_h} we obtain
\begin{align}
f(\xx_t) + g(\xx_t) - f(\xx) - g(\xx) &\leq \frac{1}{\gamma}\langle \zz_t  - \xx_t, \xx_t - \xx\rangle+ \frac{L}{2}\|\zz_t - \xx_t\|^2 - \langle \uu_t, \xx_t - \xx \rangle \nonumber \\
&\qquad + \langle \vv_t - \nabla f(\zz_t), \xx_t - \xx\rangle\label{eq:primal_subopt0}
\end{align}
Using these, we can now write the following sequence of inequalities for the Lagrangian suboptimality
\begin{align}
\mathcal{L}(\xx_t, \uu_t) - \mathcal{L}(\xx, \uu_t) &= f(\xx_t) - f(\xx) + g(\xx_t) - g(\xx) + \langle \xx_t - \xx, \uu_t\rangle\\
&\stackrel{\eqref{eq:primal_subopt0}}{\leq}  \frac{1}{\gamma}\langle \zz_t  - \xx_t, \xx_t - \xx\rangle+ \frac{L}{2}\|\zz_t - \xx_t\|^2  + \langle \vv_t - \nabla f(\zz_t), \xx_t - \xx\rangle\label{eq:primal_suboptimality}\\
\mathcal{L}(\xx_t, \uu) - \mathcal{L}(\xx_t, \uu_t) & = h^*(\uu_t) - h^*(\uu) + \langle \xx_t, \uu - \uu_t\rangle \stackrel{\eqref{eq:ineq_g}}{\leq} \langle \zz_t - \xx_t, \uu_t - \uu\rangle\label{eq:dual_suboptimality}
\end{align}
Adding these two last equations we have
\begin{align}
  &\mathcal{L}(\xx_t, \uu) - \mathcal{L}(\xx, \uu_t) \leq \frac{1}{\gamma}\langle\zz_t - \xx_t, \xx_t- \xx\rangle+ \frac{L}{2}\|\zz_t - \xx_t\|^2   + \langle \vv_t - \nabla f(\zz_t), \xx_t - \xx\rangle\nonumber\\
  &\qquad + \langle \zz_t - \xx_t, \uu_t - \uu\rangle\\
  &~= \frac{1}{\gamma}\langle\zz_t - \xx_t, \xx_t + \gamma \uu_t - \xx - \gamma \uu\rangle + \frac{L}{2}\|\zz_t - \xx_t\|^2  + \langle \vv_t - \nabla f(\zz_t), \xx_t - \xx\rangle\\
  &~=\frac{1}{\gamma}\langle \boldsymbol \yy_t - \yy_{t+1}, \yy_{t+1} - \yy \rangle + \frac{L}{2}\|\yy_t - \yy_{t+1}\|^2  + \langle \vv_t - \nabla f(\zz_t), \xx_t - \xx\rangle\\
  &\qquad \text{(using $\yy_{t+1} - \yy_t = \xx_t - \zz_t$)}\nonumber\\
  &~=\frac{1}{2\gamma}\|\yy_t - \yy\|^2 + \left( \frac{L}{2} - \frac{1}{2\gamma}\right)\|\yy_{t+1} - \yy_t\|^2 - \frac{1}{2\gamma}\|\yy_{t+1} - \yy\|^2 \nonumber\\
  &\qquad  + \langle \vv_t - \nabla f(\zz_t), \xx_t - \xx\rangle\quad, \\
  &~\leq \frac{1}{2\gamma}\|\yy_t - \yy\|^2 - \frac{1}{2\gamma}\|\yy_{t+1} - \yy\|^2   + \langle \vv_t - \nabla f(\zz_t), \xx_t - \xx\rangle\quad,
  \label{eq:square_identity}
\end{align}
where the second equality comes from the definition of $\uu_t, \yy_{t+1}$ and $\yy$ and in the last equality we have applied the identity $2\langle \boldsymbol a, \boldsymbol b\rangle = {\|\boldsymbol{a} + \boldsymbol{b}\|^2} - \|\boldsymbol a\|^2 - \|\boldsymbol b\|^2$. In the last inequality we have used the assumption $\gamma \leq 1/L$.

We will now upper bound the last term. For this, we introduce the variable $\widetilde{\xx}$, which represents the step in $\xx$ that would be taken if we used the full gradient rather than the \SAGA\ gradient approximation:
\begin{equation}
\widetilde{\xx}_t \defas \prox_{\gamma g}(2 \zz_t - \yy_t - \gamma \nabla f(\zz_t)) \quad.
\end{equation}
Taking expectations on this last quantity we have
\begin{align}
  \Econd\langle \vv_t - \nabla f(\zz_t), \xx_t - \xx \rangle
  &=  \Econd\langle \vv_t - \nabla f(\zz_t), \xx_t - \widetilde{\xx}_t \rangle + \Econd\langle \vv_t - \nabla f(\zz_t), \widetilde{\xx}_t - \xx_t \rangle \\
  &\leq \Econd\|\vv_t - \nabla f(\zz_t)\| \|\xx_t - \widetilde{\xx}_t\| + \Econd\langle \vv_t - \nabla f(\zz_t), \widetilde{\xx}_t - \xx_t \rangle \qquad \\
  &\qquad\text{ (Cauchy-Schwarz)}\nonumber\\
  &=  \Econd\|\vv_t - \nabla f(\zz_t)\| \|\xx_t - \widetilde{\xx}_t \|\\
  &  \qquad \text{ (since $\widetilde{\xx}_t$ does not depend on $i$ and $\Econd \vv_t = \nabla f(\zz_t)$)}\nonumber\\
  &=  \Econd\|\vv_t - \nabla f(\zz_t)\| \|\prox_{\gamma g}(2 \zz_t - \yy_t - \gamma \vv_t)\nonumber \\
  &\qquad\qquad\qquad\qquad- \prox_{\gamma g}(2 \zz_t - \yy_t - \gamma \nabla f(\zz_t)) \|\\
  &\leq \gamma \Econd \|\vv_t - \nabla f(\zz_t)\|\|\vv_t - \nabla f(\zz_t)\|^2\\
  &\qquad\qquad \text{ (nonexpansiveness of $\prox$)}\nonumber\\
  &= \gamma \Econd \|\vv_t - \nabla f(\zz_t)\|^2 \\
  &\leq \gamma (1 + \eta) \Econd\|\nabla\xi_i (\zz_t) - \nabla\xi_i(\xx^\star)\|^2 + 2 \gamma (1+\eta^{-1}) L_{f}H_t~,
\end{align}
where the last inequality follows by Lemma~\ref{lemma:bound_variance_2} for dense update variants.
Taking conditional expectations in \eqref{eq:square_identity}, plugging this bound, multiplying everything by $2\gamma$ and rearranging we obtain
\begin{equation}
\begin{aligned}
&2 \gamma \Econd(\mathcal{L}(\xx_t, \uu) - \mathcal{L}(\xx, \uu_t)) + \Econd\|\yy_{t+1} - \yy\|^2 \\
&\qquad \leq \Econd\|\yy_{t} - \yy\|^2 + 2\gamma^2 (1 +  \eta) \Econd\|\nabla\xi_i (\zz_t) - \nabla\xi_i(\xx^\star)\|^2 + 4\gamma^2 (1+ \eta^{-1})L_{f}H_t
\end{aligned}
\end{equation}
which is the desired bound.
\end{proof}

\begin{mdframed}
\begin{customtheorem}{\ref{thm:sublinear_convergence}}
Let $\overline{\xx}_t$ denote the averaged (also known as ergodic) iterate, i.e., $\overline{\xx}_t = {(\sum_{k=0}^t\xx_k )/(t+1)}$ and $\overline{\uu}_t = (\sum_{k=0}^t\uu_k )/(t+1)$. Then \VRTOS\ (Algorithm~\ref{alg:vrtos}) methods converge for any step size $\gamma \leq 1 / (3L_f)$, and for $\gamma = 1/ (3L_f), t \geq 0$ we have the following bound for all $(\xx, \uu) \in \dom g \times \dom h^*$:
\begin{equation}
\EE\left[\mathcal{L}(\overline\xx_t, \uu) - \mathcal{L}(\xx, \overline\uu_t)\right] \leq \frac{10 n }{q(t+1)}\left[\frac{3 L_{f}q}{20 n}\|\yy_0 - \yy\|^2 + \frac{3 L_{f}q}{2n}\|\yy_0 - \yy^\star\|^2 + H_0  \right]~,
\end{equation}
with $\yy = \xx + \gamma \uu$, $\yy^\star \in \Fix(\boldsymbol{G}_\gamma)$.

Furthermore, if $h$ is $\beta_h$-Lipschitz we have the following rate in terms of the primal objective:
\begin{equation}
     \mathcal{P}(\overline{\xx}_t) - \mathcal{P}(\xx^\star) \leq \frac{10 n }{q(t+1)}\left[\frac{6 L_{f}q}{20 n}\|\zz_0 - \xx^\star\|^2 + \frac{3 L_{f}q}{2n}\|\yy_0 - \yy^\star\|^2 + \frac{q}{15 n L_f}\beta_h^2 + H_0  \right]~.
\end{equation}
\end{customtheorem}
\end{mdframed}
\begin{proof}
We define the following Lyapunov function:
\begin{equation}
    W_t(\xx, \uu) \defas V_t + \lambda \|\yy_t - \yy\|^2 ~\text{ with $\yy = \xx + \gamma \uu$}~.
\end{equation}
We will now aim to bound $\Econd W_{t+1} - W_t$
 by using Lemma \ref{lemma:lyapunov} with $\mu=0$ and $s=1$, we have for $\Econd V_{t+1}$ that
\begin{equation}\label{eq:sublinear_vt}
    \begin{aligned}
        \Econd V_{t+1} &\leq V_t  +  \left({2L_{f}c\gamma^2}(1 + \beta^{-1}) - \frac{q}{n}\right) H_t \\
        &\qquad + c\left({\gamma^2}(1 + \beta) - \frac{\gamma}{L}\right)\Econd\| \nabla\xi_i(\zz_t) - \nabla \xi_i(\xx^\star)\|^2\\
&\qquad + \left(- 2c\gamma + \frac{q}{n}\right) B_f(\zz_t,\xx^\star)
    \end{aligned}
\end{equation}

while for the last term from Lemma~\ref{theorem:saddle_point_inequality_1} we have
\begin{equation}\label{eq:sublinear_y}
\begin{aligned}
\Econd\|\yy_{t+1} - \yy\|^2 &\leq \|\yy_t - \yy\|^2 - 2 \gamma\Econd (\mathcal{L}(\xx_t, \uu) - \mathcal{L}(\xx, \uu_t))\\
&\qquad\qquad + 2 \gamma^2 (1 + \eta) \Econd\|\nabla\xi_i (\zz_t) - \nabla\xi_i(\xx^\star)\|^2 + 4 \gamma^2(1+\eta^{-1}) L_{f}H_t
\end{aligned}
\end{equation}
Adding \eqref{eq:sublinear_vt} and \eqref{eq:sublinear_y} times $\lambda$ we have
\begin{equation}
\begin{aligned}
    &\Econd W_{t+1}(\xx, \uu) -W_{t}(\xx, \uu) \leq  - 2 \lambda \gamma\Econd (\mathcal{L}(\xx_t, \uu) - \mathcal{L}(\xx, \uu_t)) \\
    &\qquad \qquad +  \left[4\gamma^2 \lambda (1 + \eta^{-1})L_{f}+ {2L_f c\gamma^2}(1 + \beta^{-1}) - \frac{q}{n}\right] H_t\\
    &\qquad \qquad \qquad+ \gamma\left[2\lambda(1 + \eta)\gamma + {\gamma}(1 + \beta)c - \frac{c}{L_f}\right]\Econd\| \nabla\xi_i(\zz_t) - \nabla \xi_i(\xx^\star)\|^2\\
    &\qquad \qquad \qquad \qquad+ \left[- 2c\gamma + \frac{q}{n}\right] B_f(\zz_t,\xx^\star)
\end{aligned}
\end{equation}
We can now verify that with the coefficients
\begin{equation}
\gamma = \frac{1}{3 L_f},\quad c= \frac{3 L_{f}q}{2 n},\quad\beta = \eta = \frac{3}{2}, \quad\lambda = \frac{3 L_{f}q}{20 n}~,
\end{equation}
all the square brackets are negative and so we have
\begin{equation}
    \Econd W_{t+1}(\xx, \uu) -  W_{t}(\xx, \uu) \leq- \frac{q}{10 n} \Econd (\mathcal{L}(\xx_t, \uu) - \mathcal{L}(\xx, \uu_t))
\end{equation}
These expectations are conditional on information from step $t$. Taking full expectations (with respect to all randomness) we have
\begin{equation}
    \EE W_{t+1}(\xx, \uu) - \EE  W_{t}(\xx, \uu)\leq - \frac{q}{10 n} \EE\left[\mathcal{L}(\xx_t, \uu) - \mathcal{L}(\xx, \uu_t) \right]~,
\end{equation}
where all expectations are unconditional. Adding the previous inequality from $0$ to $t$, the terms in $W_{t}$ cancel each other and we have
\begin{equation}
    \frac{q}{10 n}\EE \left[\sum_{k=0}^{t} \mathcal{L}(\xx_k, \uu) - \mathcal{L}(\xx, \uu_k) \right] \leq  W_0(\xx, \uu) - \EE W_{t+1}(\xx, \uu)~.
\end{equation}
We can drop the last term since it is always negative. Note that the function $ \mathcal{L}(\xx_t, \uu) - \mathcal{L}(\xx, \uu_t)$ is convex in $\xx_t$ and $\uu_t$ and so we can apply Jensen's inequality. This gives
\begin{equation}
\begin{aligned}
    \EE \left[ \mathcal{L}(\overline\xx_t, \uu) - \mathcal{L}(\xx, \overline\uu_k) \right] &\leq  \frac{10 n}{q(t+1)}W_0(\xx, \uu)\\
    &= \frac{10 n }{q(t+1)}\left[\frac{3 L_{f}q}{20 n}\|\yy_0 - \yy\|^2 + \frac{3 L_{f}q}{2n}\|\yy_0 - \yy^\star\|^2 + H_0  \right]~,
\end{aligned}
\end{equation}
which proves the first result of the theorem.

For the second result, let $\widehat\uu \defas \argmin_{\uu} \mathcal{L}(\overline\xx_{t+1}, \uu)$ and $(\xx^\star, \uu^\star)$ be a saddle point of $\mathcal{L}$. Then $\mathcal{L}(\overline\xx_{t+1}, \widehat\uu) = P(\overline\xx_{t+1})$ and $\mathcal{L}(\xx^\star, \uu^\star) = P(\xx^\star)$ by definition of Fenchel dual.

At the same time, by the $\beta_h$-Lipschitz assumption on $h$ implies that the norm of every element in $\dom h^*$ is bounded by $\beta_h$ (see e.g., \citep[Corollary 13.3.3]{rockafellar1997convex}). This way we bound
\begin{align}
\|\yy_0 - \yy\|^2 &= \|\zz_0 + \gamma \uu_0 - \yy\|^2 \leq 2\|\zz_0 - \xx\|^2 + 2 \gamma^2\|\uu_0 - \uu\|^2\\
&\leq 2 \|\zz_0 - \xx\|^2 + 4 \gamma^2 \beta_h^2
\end{align}
Plugging this bound into the last inequality with $\xx = \xx^\star$
\begin{equation}
  P(\overline\xx_{t+1}) - P(\xx^*) \leq \frac{10 n }{q(t+1)}\left[\frac{6 L_{f}q}{20 n}\|\zz_0 - \xx^\star\|^2 + \frac{3 L_{f}q}{2n}\|\yy_0 - \yy^\star\|^2 + \frac{q}{15 n L_f}\beta_h^2 + H_0  \right]~.
\end{equation}

\end{proof}

\clearpage
\subsection{Sublinear convergence -- sparse algorithms}\label{apx:sublinear_sparse}

\begin{mdframed}
\begin{customtheorem}{\ref{thm:sublinear_convergence_sparse}}
\TheoremSubSparse
\end{customtheorem}
\end{mdframed}
\begin{proof}
    Using the Lyapunov inequality of Lemma~\ref{lemma:lyapunov} for non-strongly convex functions, i.e., with $\mu=0$ we have
\begin{equation}\label{eq:proof_sublinear_0}
    \begin{aligned}
        \Econd V_{t+1} &\leq V_t + c(s - 1)\Econd \|\yy_{t+1} - \yy_t\|^2 +  \left(\frac{2L_f c\gamma^2}{s}(1 + \beta^{-1}) - \frac{q}{n}\right) H_t \\
        &\qquad + c\left(\frac{\gamma^2}{s}(1 + \beta) - \frac{\gamma}{L_f}\right)\Econd\| \nabla\xi_i(\zz_t) - \nabla \xi_i(\xx^\star)\|^2\\
&\qquad \qquad + \left( - 2c\gamma + \frac{q}{n}\right) B_f(\zz_t,\xx^\star)
    \end{aligned}
\end{equation}
where $V_t$ and $H_t$ are as defined in Eq. \eqref{eq:def_lyapunov}. For notational convenience, we define $\boldsymbol{R}$ as the operator residual $\boldsymbol{R}(\yy) = \boldsymbol{G}_\gamma(\yy) - \yy$, and denote by $i$ the random index selected at the $t$-th iteration.
The term $\|\yy_{t+1} - \yy_t\|^2$ can be bounded in term of the gradient mapping using the following inequality, where $\widetilde{\xx} = \prox^{\DD^{-1}}_{\gamma h}(2 \zz_t - \yy_t - \DD \nabla f(\zz_t))$ is the value of $\xx_t$ had we used the full gradient instead of the stochastic approximation:
  \begin{align}\label{eq:proof_sublinear_1}
      \|\PP_i\boldsymbol{R}(\yy_t)\|^2 &= \|\yy_{t+1} - \yy_t + \PP_i\boldsymbol{R}(\yy_t) - \yy_{t+1} + \yy_t\|^2 \\
      &\leq 2 \|\yy_{t+1} - \yy_t\|^2 + 2\|\PP_i\boldsymbol{R}(\yy_t) - \yy_{t+1} + \yy_t\|^2 \\
      &= 2 \|\yy_{t+1} - \yy_t\|^2 + 2\|\boldsymbol{R}(\yy_t) - \yy_{t+1} + \yy_t\|_{(i)}^2 \\
      &\qquad \text{(since both $\PP_i\boldsymbol{R}(\yy_t)$ and $\yy_{t+1} + \yy_t$ have support in $T_i$)}\nonumber\\
      &= 2 \|\yy_{t+1} - \yy_t\|^2 + 2\|\boldsymbol{G}_\gamma(\yy_t) - \yy_{t+1}\|_{(i)}^2 \\
      &\qquad \text{(by definition of $\boldsymbol{R}$)}\nonumber\\
      &= 2 \|\yy_{t+1} - \yy_t\|^2 + 2\|\yy_t - \zz_t + \widetilde{\xx}_t - (\yy_{t} - \zz_t + \xx_t)\|_{(i)}^2 \\
      &\qquad \text{(by definition of $\boldsymbol{G}_\gamma$ and $\yy_{t+1}$)}\nonumber\\
      &= 2 \|\yy_{t+1} - \yy_t\|^2 + 2 \|\widetilde{\xx}_t - \xx_t\|_{(i)}^2 ~.
  \end{align}
  For the last term, we further have
    \begin{align}
      \|\xx_t - \widetilde{\xx}_t\|_{(i)}^2 &= \|\prox^{\DD^{-1}}_{\gamma h}(2 \zz_t - \yy_t - \gamma \vv_t) - \prox^{\DD^{-1}}_{\gamma h}(2 \zz_t - \yy - \gamma \DD \nabla f(\zz_t))\|_{(i)}^2\\
      &\leq \gamma^2\|\vv_t - \DD \nabla f(\zz_t)\|^2_{(i)}\qquad \text{ (by Lemma~\ref{lemmma:block_nonexpansive})} \\
    &\leq 2 \gamma^2 \Econd \|\nabla \xi_i(\zz_t) - \nabla \xi_i(\xx^\star)\|^2 + 4 \gamma^2 L_f  H_t\quad\text{ (by Lemma \ref{lemma:bound_variance_2})}
  \end{align}
  Combining this into Eq.~\eqref{eq:proof_sublinear_1} and tacking expectation, we have:
  \begin{align}
    &\Econd\|\PP_i\boldsymbol{R}(\yy_t)\|^2 \leq 2 \Econd\|\yy_{t+1} - \yy_t\|^2 +  4 \gamma^2 \Econd \|\nabla \xi_i(\zz_t) - \nabla \xi_i(\xx^\star)\|^2 + 8 \gamma^2 L_f  H_t \\
    \iff - &\Econd \|\yy_{t+1} - \yy_t\|^2 \leq - \frac{1}{2}\Econd\|\PP_i\boldsymbol{R}(\yy_t)\|^2+ 2 \gamma^2 \Econd \|\nabla \xi_i(\zz_t) - \nabla \xi_i(\xx^\star)\|^2 + 4 \gamma^2 L_f  H_t
  \end{align}
Plugging this last inequality in Eq.~\eqref{eq:proof_sublinear_0} gives
\begin{equation}
    \begin{aligned}
        \Econd V_{t+1} &\leq V_t + \frac{c(s - 1)}{2}\Econd\|\PP_i\boldsymbol{R}(\yy_t)\|^2 +  \left[2 c (s-1)\gamma^2 L_f + \frac{2L_f c\gamma^2}{s}(1 + \beta^{-1}) - \frac{q}{n}\right] H_t \\
        &\qquad + c\left[ (s-1)\gamma^2 + \frac{\gamma^2}{s}(1 + \beta) - \frac{\gamma}{L_f}\right]\Econd\| \nabla\xi_i(\zz_t) - \nabla \xi_i(\xx^\star)\|^2\\
&\qquad \qquad + \left[ - 2c\gamma + \frac{q}{n}\right] B_f(\zz_t,\xx^\star)
    \end{aligned}
\end{equation}
We can verify that with the following values
\begin{equation}
    \gamma = \frac{1}{3 L_f}~,\quad \beta = 3/2~,\quad s = 8/10~,\quad c = \frac{2 L q}{n}~,
\end{equation}
all the square brackets in the previous expression are non-positive and so we have
\begin{align}
  &\Econd V_{t+1} - V_t \leq- \frac{L q}{5 n}\Econd \|\PP_i \boldsymbol{R}(\yy_t)\|^2\\
  &\qquad\qquad = - \frac{L q}{5 n}\Econd \|\boldsymbol{R}(\yy_t)\|_{(i)}^2 = - \frac{L q}{5 n} \|\boldsymbol{R}(\yy_t)\|_{\DD^{-1}}^2\\
  &\qquad\qquad \leq - \frac{L q}{5d_{\max} n  } \|\boldsymbol{R}(\yy_t)\|^2\\
  &\iff V_t  - \Econd V_{t+1} \geq \frac{ L q}{5 d_{\max} n} \|\boldsymbol{R}(\yy_t)\|^2
\end{align}
Summing from 0 to $t$ and chaining expectations have
$$
\EE V_0 - \EE V_{t+1} \geq \frac{L q}{5 d_{\max} n} \sum_{k=0}^{k}\|\boldsymbol{R}(\yy_k)\|^2 \geq  \frac{L q}{5 d_{\max} n} \sum_{k=0}^{t}\|\boldsymbol{R}(\yy_k)\|^2 \geq  \frac{L q (t+1)}{5 d_{\max} n }\min_{k=0, \ldots, t} \|\boldsymbol{R}(\yy_t)\|^2
$$
Dropping $\EE V_{t+1}$ (since it is positive)  and taking the square root we have
\begin{equation}
    \min_{k=0, \ldots, t} \|\boldsymbol{R}(\yy_t)\|^2 \leq \frac{5 d_{\max} n}{L q (t+1)} V_0~,
\end{equation}
The final results follows then by definition of $\boldsymbol{R}$.
\end{proof}

\clearpage

\section{Learning with multiple penalties}\label{apx:optimizing_multiple_penalties}

In this section we review some cases in which we can compute the scaled proximal operator $\prox_{\gamma h}^{\DD^{-1}}$ for some diagonal matrix $\DD$. We refer to \citep{pedregosa2018adaptive} for a discussion on how common penalties such as $\ell_1$ trend filtering, multidimensional total variation, overlapping group lasso, etc. can be split as a sum of proximal terms.

\subsection{$\ell_1$ norm}

We consider the case in which $g$ is the $\ell_1$ or Lasso penalty, $g(\xx) \defas \|\xx\|_1$.
Since this function is fully separable, its resolvent can be computed component-wise. Hence, the reweighting matrix $\DD$ can be associated with the step size $\gamma$ and using the known prox for the Lasso penalty we obtain
$$
\left[(\text{Id} + \gamma \DD \partial g)^{-1} \xx\right]_j = \Big(1 - \frac{[\DD]_{j, j}\gamma}{|\xx_j|}\Big)_+ \xx_j
$$

\subsection{Fused lasso}

The fused lasso penalty, also known as \mbox{1-di}mensional total variation, is defined as the $\ell_1$ norm of the differences between consecutive coefficients. Although in this case direct methods have been developed to compute its proximal operator~\citep{condat2013direct,johnson2013dynamic}, there still exist advantages in splitting the penalty.
In particular, existing direct approaches involve dense updates due to the non-separability of the penalty. However, by splitting the penalty into constituents that are block-separable, it is possible to optimize with this penalty while only performing sparse updates. The split is the following:
\begin{align}
  \|\xx\|_\text{FL} &\defas \textstyle\sum_{i=1}^{p-1} |\xx_i - \xx_{i+1}|= \underbrace{\textstyle\sum_{i=1}^{r}|\xx_{2 i - 1} - \xx_{2 i}|}_{\defas g(\xx)} + \underbrace{\textstyle\sum_{i=1}^{s}|\xx_{2 i} - \xx_{2 i+1}|}_{\defas h(\xx)},
\end{align}
with $r = \lfloor {(p-1)}/{2}\rfloor$ and $s = \lfloor {p}/{2}\rfloor$.
We note that both $g$ and $h$ are block-separable with blocks of size 2.
Furthermore, it is possible to compute the scaled proximal operator of $\prox_{\gamma g}^Q(\xx)$ in closed form.
The advantages of  \VRTOS\ with this formulation on large and sparse problems is demonstrated experimentally in \S\ref{scs:experiments}.

Both functions $g$ and $h$ are block-separable with blocks of size two. Hence it is sufficient to specify the proximal operator on a vector of size two. Let $\xx = (\xx_1, \xx_2)$ and $\DD = \text{diag}(\qq_1, \qq_2)$. Then we have
\begin{equation}
\prox^{\DD^{-1}}_{\gamma g}(\xx) = \begin{cases}
\big(\xx_1 - \nicefrac{\gamma}{\qq_1}, \xx_2 + \nicefrac{\gamma}{\qq_2}\big) &\!\!\text{if } \xx_1 - \gamma/ \qq_1 \geq \xx_2 + \gamma \qq_2\\
\big(\xx_1 + \nicefrac{\gamma}{\qq_1}, \xx_2 - \nicefrac{\gamma}{\qq_2}\big) &\!\!\text{if } \xx_1 + \gamma/\qq_1 \leq \xx_2 - \gamma \qq_2\\
\big(\frac{\qq_1 \xx_1 + \qq_2 \xx_2}{\qq_1 + \qq_2}, \frac{\qq_1 \xx_1 + \qq_2 \xx_2}{\qq_1 + \qq_2}\big)  &\!\!\text{otherwise }~.
\end{cases}
\end{equation}
\begin{proof}
  Let $(\zz_1, \zz_2) = \prox^{\DD^{-1}}_{\gamma g}(\xx_1, \xx_2)$. The first order optimality conditions applied to this problem give
  $$
  \frac{\DD^{-1}}{\gamma}((\xx_1, \xx_2) - (\zz_1, \zz_2)) \in \partial |\zz_1 - \zz_2|
  $$
  We now perform a dichotomy of cases. Suppose first $\zz_1 - \zz_2 > 0$. Then the above becomes
  $$
  \frac{\DD^{-1}}{\gamma}((\xx_1, \xx_2) - (\zz_1, \zz_2)) = (1, -1)
  $$
  from where the solution is given by $ (\xx_1 - {\gamma}/\qq_1, \xx_2 + {\gamma}/\qq_2)$, but only if $\xx_2 - \gamma/\qq_1 \geq \xx_2 + \gamma/\qq_2$, otherwise the assumption $\zz_1 - \zz_2 < 0$ would be violated.

  Repeating this for $\zz_1 - \zz_2 < 0$ and $\zz_1 - \zz_2=0$ yields the above rule.
\end{proof}

%

\clearpage

\section{Pseudocode for the extension to $k$ proximal terms}\label{scs:implementation_extension}

The extension of the proposed method to $k$ proximal terms consists in running Algorithm~\ref{alg:vrtos} or \ref{alg:vrtos_sparse} on particular values of $g$ and $h$. Some tricks can help to reduce the memory usage of this algorithm, reducing the storage of vectors $\xx, \zz$ and $\vv_t$ from $k \times p$ to $p$. In this  subsection we provide the pseudocode for runnning Sparse \VRTOS\ on its $k$-proximal terms extension.

As in \S\ref{scs:extension_k_terms} we consider an optimization problem of the form
\begin{empheq}[box=\mybluebox]{equation}\tag{OPT-$k$}
\begin{aligned}
    &\minimize_{\XX \in \RR^{k \times p}}\, f(\overline{\XX}) + \textstyle\sum_{j=1}^k g_j(\XX_j) + h(X)\,,\nonumber \\
    &\text{ with } f(\xx) = \textstyle\frac{1}{n} \sum_{i=1}^n \psi_i(\xx) + \omega(\xx)~,
\end{aligned}
\end{empheq}
where $h(X) = \imath\{\XX_1 = \cdots = \XX_k\}$. We will first detail how the scaled proximal operator of $h$ can be computed

\begin{lemma}\label{lemma:projection_kterms}
Let $h(X) = \imath\{\XX_1 = \cdots = \XX_k\}$. Then we have that
\begin{gather}
\prox_{\gamma h}^{\DD^{-1}}(\xx) = \zz {\mathbf{1}_k}^T \text{ for $\zz \in \RR^p$ defined as }\\
 \zz_j = \left(\sum_{i=1}^k a_{i, j} \XX_{i, j}\right) / \left(\sum_{i=1}^k a_{i, j}\right) \text{ with  $a_{i, j} = \DD^{-1}_{i p + j, i p + j}$.}
\end{gather}
\end{lemma}
\begin{proof}

Let $\mathcal{S}$ denote the domain of $h$, i.e., $\mathcal{S} \defas \{\XX \in \RR^{k \times p}| \XX_1\!=\!\XX_2\!=\!\cdots\!=\!\XX_k\}$. Computing this proximal operator consists by definition of scaled proximal operator in solving the following optimization problem
\begin{align}
  &\argmin_{\ZZ \in S} \|\vec(\ZZ) - \vec(\XX)\|_{\DD^{-1}}^2 = \argmin_{\zz \in \RR^{p}} \|\vec(\zz \boldsymbol{1}_k^T) - \vec(\XX)\|^2_{\DD^{-1}}
\end{align}
The problem is then separable along the components of $\zz$, and the $j$-th component is the solution to the problem
\begin{equation}
\argmin_{\zz_j \in \RR} \sum_{i=1}^k a_{i, j} (\zz_j -  \XX_{j, i})^2 \text{ with $a_{i, j} = \DD^{-1}_{i p + j, i p + j}$}~,
\end{equation}
and whose solution is
\begin{equation}
    \zz_j = \left(\sum_{i=1}^k a_{i, j} \XX_{i, j}\right) / \left(\sum_{i=1}^k a_{i, j}\right)
\end{equation}
\end{proof}

Before introducing the algorithm, we make the following definitions:
\begin{itemize}
    \item Let $\mathcal{B}_j$ denote the blocks of $g_j$, that is, $g_j$ can be decomposed block coordinate-wise as $g_j(\xx) = \sum_{B \in \mathcal{B}_j} g_{j, B}([\xx]_B)$.
    \item Let $T_{i, j}$ denote the extended support of $\nabla \psi_i$ in $\mathcal{B}_j$, that is, $T_{i, j} \defas \{B: \text{supp}(\nabla f_i) \cap B \neq \varnothing, \,B\in\mathcal{B}_j \}$.
    \item Let $S_i$ be the set of coordinates that are at least in one block of one of the extended supports, that is, $S_i \defas \{c: c \in B \text{ for any } B \in T_{i, j} \text{ and any } j=1, \ldots, k\}$.
\end{itemize}

With respect to Algorith~\ref{alg:vrtos_sparse}, compute the $\zz$ update at the end of the algorithm instead of the beginning to efficiently use the extended support.

\begin{algorithm}[h]
 \KwIn{$\YY_0 \in \RR^{k \times p}$, $\balpha_0 \in \RR^{n \times p}$, $\gamma > 0$}

 {\bfseries Temporary storage}: $\zz_t$, $\vv_t$ and $\xx_t$, all in $\RR^p$

 \KwResult{approximate solution to \eqref{eq:obj_fun_k} }

\For{$t=0, 1, \ldots $ }{

Sample $i \in \{1, \ldots, n\}$ uniformly at random

Compute $\nabla \psi_i(\zz_t)$

\For{$j=1, \ldots, k$}{
$[\vv_{t}]_{T_{i, j}} = \frac{1}{k}[\nabla \psi_i(\zz_t)\!-\!\balpha_{i, t} + \DD^{(j)}(\overline{\balpha}_t + \nabla \omega(\zz_t))]_{T_i}$

$[\xx_t]_{T_{i, j}} = [\prox_{\gamma \varphi_{i, j}}(2 \zz_t - \yy_t - \gamma \vv_t)]_{T_i}$

$[\YY_{j, t+1}]_{{T_{i, j}}} = [\YY_{j, t} + \xx_t - \zz_t]_{T_i}$

\For{$b \in T_i$}{
$    \zz_{t+1, b} = \left(\sum_{l=1}^n \DD^{(l)}_{b, b} \YY_{l, b}\right) / \left(\sum_{l=1}^n \DD^{(l)}_{b, b}\right)$
}
}

update $\balpha_{t+1}$ according to \eqref{eq:q_memorization}
}

\Return $\prox^{\DD^{-1}}_{\gamma h}(\yy_t)$

 \caption{Sparse \VRTOS\ for $k$ proximal terms}\label{alg:vrtos_sparse_k}
\end{algorithm}

\vspace{1em}

\clearpage

\section{Experiments}\label{apx:experiments}

\subsection{Implementation aspects}\label{apx:implementation}

We review some implementation details for the proposed algorithms

\paragraph{Update of memory terms.} In a practical implementation of the \SAGA\ variants, the vector $\overline{\balpha}_t= ({1}/{n})\sum_{i=1}^n\balpha_{i, t}$ is also stored in memory and updated incrementally as $\overline\balpha_{t+1} = \overline\balpha_t + (\balpha_{i, t+1} - \balpha_{i, t})/n$.

\paragraph{Compressed memory storage.} Like other \SAGA\ variants, \VRTOS\ with the \SAGA-like update of memory terms requires to store a table of partial gradients. In the general case, this requires a matrix of size $n \times p$. However, for linearly-parametrized loss functions this can be compressed into a matrix of size $n$. Linearly-parametrized functions are of the form $\psi_i(\xx) = l_i(\aa_i^T \xx)$ for some input dataset $\{\aa_i\}_{i=1}^n$ and some real functions $\{l_i\}_{i=1}^n$. Deriving with respect to $\xx$ one obtains $\nabla \psi_i(\xx) = \aa_i l_i'(\aa_i^T \xx)$. In this expression only the factor $l_i'(\aa_i^T \xx)$ depends on the iterate $\xx$, and it is a scalar. Hence, we only need to store this scalar and we can construct the partial gradient at run time by multiplying by the vector $\aa_i$. The memory cost is hence reduced to a list of $n$ scalars.

\paragraph{Initialization of $\balpha_0$.} The original \SAGA\ algorithm of~\citep{defazio2014saga} required to initialize the memory terms as $\balpha_{i, t}= \nabla \psi_i(\zz_0)$. This is no longer required in our algorithm, in which these memory terms can be initialized arbitrarily. In fact, we recommend to initialize them to zero. This is convenient and makes the gradient estimate $\vv_t$ close to the  \SGD\ estimate during the first iterations.

\paragraph{Initialization of $\yy_0$.} An ``initial guess'' $\yy_0$ must also be provided. From Theorem~\ref{thm:sublinear_convergence} and \ref{apx:fixed_point_characterization}, we have that $\yy_t$ converges towards $\xx^\star + \gamma \uu^\star$, where $\uu^\star$ is a minimizer of the dual objective $\mathcal{D}$. Hence, the ideal initialization for this vector is $\xx_0 + \gamma \uu_0$, where $\xx_0$
is an initial guess for \eqref{eq:opt_problem} and $\uu_0$ is an initial guess for the dual problem. However, we rarely have an initial guess for the dual problem, in which case one can set $\uu_0 = \boldsymbol{0}$.

\paragraph{\textsc{Sgd-Tos}.} Following \citep{yurtsever2016stochastic}, we used a step size of the form $\gamma/t$ in this case, where $t$ is the number of iterations.

\paragraph{Software.} All methods are implemented in Python. Numba was used to speed up the inner loops of stochastic methods (\VRTOS, \SAGA, \ProxSVRG\ and \textsc{Stos}).
For the Adaptive Three Operator splitting method we used the implementation provided by the authors\footnote{http://openopt.github.io/copt/}.

\subsection{Overlapping Group Lasso Benchmarks}\label{apx:overlapping_benchmarks}

In this subsection we giver some details on the benchmarks reported in \S\ref{scs:experiments} that were omitted from the main text.

The associated objective function that we consider is
\vspace{-0.0em}$$
\frac{1}{n} \sum_{i=1}^n \log\big(1 + \exp(- {b}_i \aa_i^\intercal \xx)\big) + \frac{\lambda_1}{2} \|\xx\|^2 + \lambda_2 \|\xx\|_\text{OGL},
\vspace{-0.0em}$$
where $\aa_i \in \mathbb{R}^p$ and $b_i \in \{-1,+1\}$ are the data samples.

The overlapping group lasso penalty $\|\cdot\|_\text{OGL}$ is defined as the sum over the group norms. Given a collection of (potentially overlapping) groups $\mathcal{G}$, the overlapping group penalty is given by
\begin{equation}
    \|\xx\|_\text{OGL} = \sum_{g \in \mathcal{G}}\|[\xx]_g\|_2~.
\end{equation}

In our comparison the groups are chosen to have 10 variables with 2 variables of overlap between two successive groups: $\{\{1, \ldots, 10\}, \{8, \ldots, 18\}, \{16, \ldots, 26\}, \ldots\}$.

Although this penalty can be expressed as a sum of only two proximal terms, we instead use the formulation in \S\ref{scs:extension_k_terms} in order to avoid computing the scaled proximal operator and to better leverage the sparsity in the dataset.

\paragraph{Extra experiments.} We also run the same benchmark on the KDD12 dataset (149,639,105 samples and 54,686,452 features) but was not shown in the main paper due to lack of space. The results are displayed below and are consistent with the rest of the experiments.

\begin{figure}[h]
    \centering
    \includegraphics[width=0.5\linewidth]{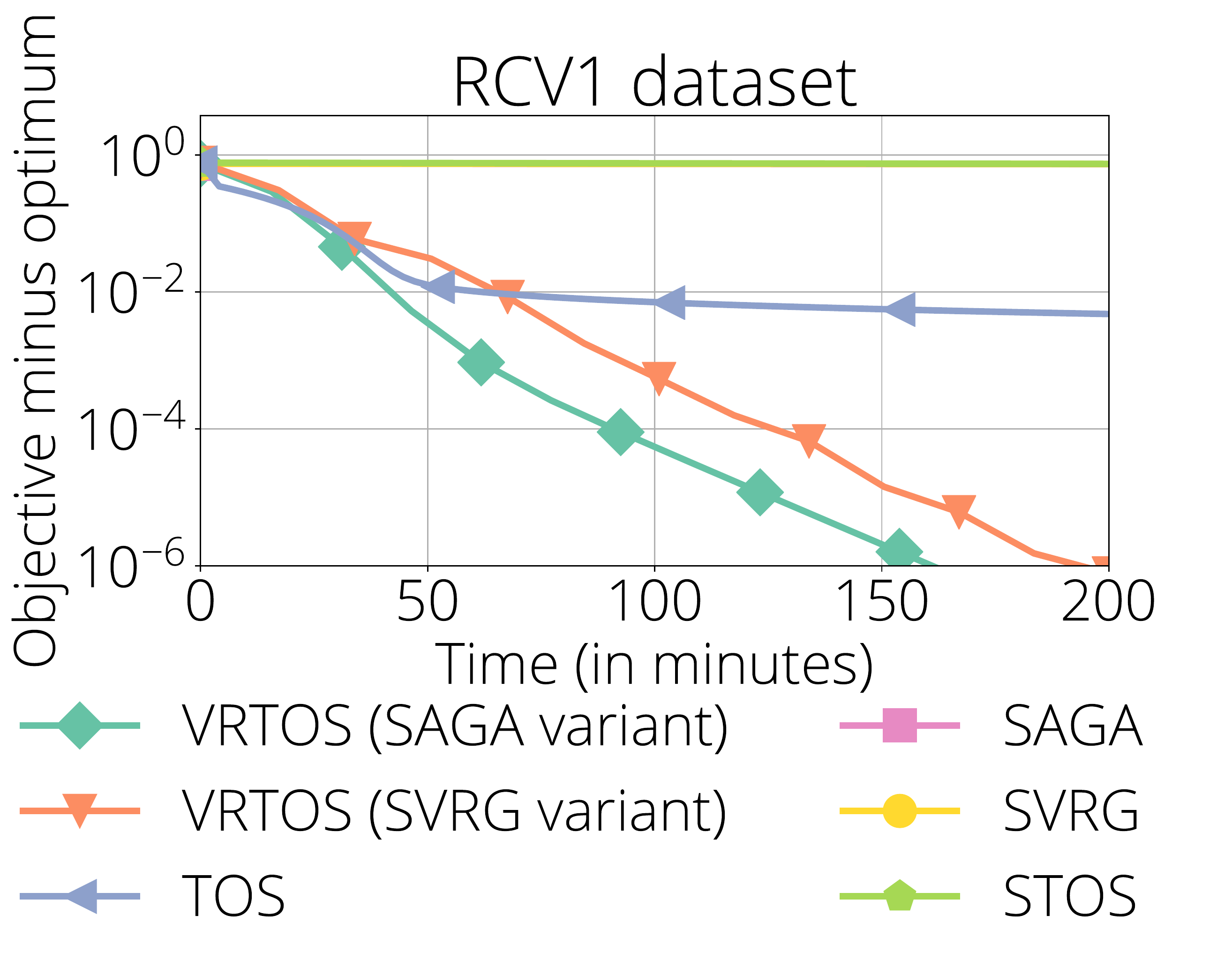}
    \caption{Benchmarks on the KDD12 dataset. }
    \label{fig:extra_exp}
\end{figure}

\end{document}